\documentclass[11pt]{elsarticle}

\usepackage{lineno,hyperref}
\modulolinenumbers[5]

\usepackage{amsmath,amssymb,amsthm,amsfonts,afterpage,float,bm,stmaryrd,paralist,wrapfig,bbm,soul}

\usepackage{cancel}
\usepackage{color}
\usepackage{fancyhdr}
\usepackage{graphics}
\usepackage{hyperref}
\usepackage{graphicx,epsf}
\usepackage{makeidx}
\usepackage{epsfig}
\usepackage{lscape}
\usepackage{upgreek}

\usepackage[margin=0.25in]{geometry}
\usepackage{pgfplots}
\usepackage{tikz}
\usetikzlibrary{arrows.meta}

\usepackage{empheq,xcolor}

\pssilent

\newtheorem{theorem}{Theorem}[section]
\newtheorem{lemma}{Lemma}[section]

\newtheorem{remark}{Remark}[section]

\numberwithin{equation}{section}
\numberwithin{figure}{section}
\numberwithin{table}{section}

\makeatletter
\def\ps@pprintTitle{%
 \let\@oddhead\@empty
 \let\@evenhead\@empty
 \def\@oddfoot{\reset@font\hfil\thepage\hfil}
 \let\@evenfoot\@oddfoot}
\makeatother

\def\XXint#1#2#3{{\setbox0=\hbox{$#1{#2#3}{\int}$}
\vcenter{\hbox{$#2#3$}}\kern-.51\wd0}}

\newcommand{\bx}{\mathbf x}
\newcommand{\by}{\mathbf y}

\newcommand{\bfm}{\mathbf m}

\newcommand{\bv}{\mathbf v}
\newcommand{\bu}{\mathbf u}
\newcommand{\bT}{\mathbf T}

\newcommand{\calC}{\mathcal{C}}

\renewcommand{\div}{\operatorname{div}}

\begin{document}

\setlength{\pdfpageheight}{\paperheight}
\setlength{\pdfpagewidth}{\paperwidth}
\title{Synchronized Optimal Transport for Joint Modeling of Dynamics Across Multiple Spaces}
\author{Zixuan Cang}
\address{Department of Mathematics, North Carolina State University, Raleigh, NC.}
\author{Yanxiang Zhao}
\address{Department of Mathematics, George Washington University, Washington D.C.}

\begin{abstract}
  Optimal transport has been an essential tool for reconstructing dynamics from complex data. With the increasingly available multifaceted data, a system can often be characterized across multiple spaces. Therefore, it is crucial to maintain coherence in the dynamics across these diverse spaces. To address this challenge, we introduce Synchronized Optimal Transport (SyncOT), a novel approach to jointly model dynamics that represent the same system through multiple spaces. With given correspondence between the spaces, SyncOT minimizes the aggregated cost of the dynamics induced across all considered spaces. The problem is discretized into a finite-dimensional convex problem using a staggered grid. Primal-dual algorithm-based approaches are then developed to solve the discretized problem. Various numerical experiments demonstrate the capabilities and properties of SyncOT and validate the effectiveness of the proposed algorithms.
\end{abstract}


\date{\today}
\maketitle

\section{Introduction}\label{section:Intro}

Optimal transport (OT) finds an optimal coupling between two distributions that minimizes the total coupling cost and provides a geometric-aware metric for distributions \cite{peyre2019computational,villani2009optimal}. OT has found broad applications in imaging science \cite{chizat2018scaling}, machine learning \cite{tolstikhin2017wasserstein, fatras2020learning}, and biology \cite{schiebinger2019optimal,cang2020inferring,CangZhao_NatureMethod2023,sha2024reconstructing,tong2020trajectorynet}. Significant theoretical and computational advances have led to efficient algorithms for practical applications of OT \cite{cuturi2013sinkhorn,ferradans2014regularized,benamou2015iterative}. Variants of conventional OT have been developed to address tasks that are commonly encountered in applied fields,  including partial OT \cite{figalli2010optimal,igbida2018augmented}, unbalanced OT \cite{chizat2018scaling,chizat2018unbalanced} and unnormalized OT \cite{gangbo2019unnormalized,lee2021generalized} that relax the marginal distribution constraints; structured OT \cite{alvarez2018structured} for structured data like graphs; and multi-marginal OT \cite{pass2015multi,strossner2023low} that connects more than two distributions.

Recently, OT has become a powerful tool for reconstructing temporal dynamics from complex data, particularly in the field of biology. Static OT has been applied to find correspondence between cells groups from high-dimensional gene expression data measured at different time points \cite{schiebinger2019optimal}. Additionally, the dynamical formulation of OT has been applied to inferring dynamics of high temporal resolution from a sequence of data measured at coarse time points \cite{sha2024reconstructing,tong2020trajectorynet}. These reconstructed temporal correspondence and high-resolution temporal data enables learning complex data-driven models and recovering underlying mechanisms from complex data. The governing rules and underlying mechanisms can be revealed using methods that infer the ground costs that lead to the observed transport plans \cite{stuart2020inverse,scarvelis2022riemannian}.

With the advancements in data acquisition technologies, multifaceted data of a same system becomes increasingly common. For example, various high-throughput approaches have been developed to investigate a biological system from different aspects, such as single-cell RNA sequencing for gene expression levels, single-cell ATAC sequencing for chromatin accessibility, and spatial transcriptomics acquiring spatial information \cite{heumos2023best}. Each modality of such data offers a unique view of the same system but from distinct spaces. While existing single-space dynamical OT methods excel in reconstructing dynamics from the high-dimensional data of one modality, they often fail to maintain coherence across multiple modalities. Optimal dynamics in one space may induce an abnormally high-cost dynamics in another. Therefore, there is an emerging need for a dynamical OT method that can jointly model the dynamics across multiple related spaces, ensuring coherent application of OT to multimodal data. 

In this work, we propose a dynamical OT framework, {\it synchronized optimal transport} (SyncOT), which simultaneously handles multiple spaces to seek coherent dynamics across multimodal data or models. In this framework, the dynamics in a designated primary space are used to induce corresponding dynamics in several secondary spaces through predefined mappings between spaces. SyncOT minimizes the aggregated kinetic energy of the dynamics in the primary space and the induced dynamics in the secondary spaces. Numerical algorithms are developed to solve the proposed SyncOT for two commonly encountered types of mappings between spaces. The framework is discretized using a staggered grid and solved using advanced primal-dual based algorithms. The properties of SyncOT and the algorithms are demonstrated with various numerical experiments, including multiple connected spaces of different dimensions and an OT problem where dynamics through consistent background environment is preferred.

\section{Background of Dynamical Optimal Transport} \label{section:Overview}

In this section, we briefly review the dynamical optimal transport. Given two probability measures $\mu\in\mathcal{P}(X), \nu\in\mathcal{P}(Y)$, and $c: X \times Y \rightarrow [0,\infty]$ lower semi-continuous, the Monge and Kantorovich OT problems are defined as follows \cite{figalli2021invitation}:
\begin{align}
&C_{M}(\mu,\nu) = \inf\left\{ \int_{X} c(x,T(x)) \ \text{d}\mu(x)\ |\  {T_{\sharp} \mu = \nu} \right\}, \label{eqn:MongeOT}\\
&C_{K}(\mu,\nu) = \inf\left\{ \int_{X\times Y} c(x,y) \ \text{d}\gamma(x,y)\ |\  \gamma \in \Gamma(\mu,\nu) \right\}. \label{eqn:KantorovichOT}
\end{align}
Here $T_{\sharp}$ is the pushforward operator, and $c(x,y)$ is the cost function. $\Gamma(\mu,\nu)$ is the set of couplings between $\mu$ and $\nu$. 

When $X=Y$, and given $\mu,\nu \in \mathcal{P}_{p}(X)$ two probability measures with finite $p$-th moment, their $p$-Wasserstein distance is defined as \cite{figalli2021invitation}
\begin{align}
  W_{p}^p(\mu,\nu) := \inf_{\gamma \in \Gamma(\mu,\nu)} \int_{X\times X} d(x,y)^p\ \text{d}\gamma(x,y).
\end{align}
When $p=2$, Benamou and Brenier \cite{benamou2000computational} showed that the squared 2-Wasserstein distance $W_2^2$ is equivalent to the following dynamical OT formulation
\begin{equation}\label{eq:DynamicalOT01}
  \begin{aligned}
  &    \min_{(\rho,\mathbf{v})}\  \int_0^1\int_{X} \rho \|\mathbf{v}\|^2 \ \mathrm{d}\mathbf{x}\mathrm{d}t, \\
  \text{s.t.} \quad & \dot{\rho} = -\nabla\cdot(\rho\mathbf{v}),\ \rho(\mathbf{x},0)=\mu,\ \rho(\mathbf{x},1)=\nu,\ \mathbf{v}\cdot \mathbf{n}|_{\partial X} = 0.
  \end{aligned}
\end{equation}
Using Otto's calculus, the dynamical OT (\ref{eq:DynamicalOT01}) can be also reformulated as
\begin{align}\label{eqn:DynamicalOT02}
  W_2^2(\mu,\nu) = \inf_{\rho \in \Gamma(\mu,\nu)} \int_0^1 \left\langle\dot{\rho}, \dot{\rho}\right\rangle_{T_{\rho}}^{W_2} \text{d}t.
\end{align}
Slightly abusing notation, we denote $\Gamma(\mu,\nu) = \left\{\rho(\cdot,t)\in\mathcal{P}(X), \forall t\ |\ \rho(\cdot,0) = \mu, \rho(\cdot,1) = \nu \right\}$ and extend the notation $\Gamma(\mu,\nu)$ to represent both the time-independent couplings $\gamma$ and the time-dependent density $\rho$ connecting $\mu$ and $\nu$.

The dynamical OT (\ref{eq:DynamicalOT01}) has been extended to unbalanced dynamical OT \cite{chizat2018unbalanced, liero2018optimal}, also called  Wasserstein-Fisher-Rao distance, by introducing additional terms allowing for birth and death of particles:
\begin{align}\label{eqn:WFR}
  \text{WFR}_{\lambda}(\mu,\nu) := \inf_{(\rho,\mathbf{v},g)} \int_0^1 \int_{X} \rho(\|\mathbf{v}\|^2 + \lambda g^2) \ \text{d}\mathbf{x}\text{d}t
\end{align}
subject to the continuity equation with a growth term $g$:
\begin{align}\label{eqn:continuity_growthterm}
  \dot{\rho} = -\nabla\cdot(\rho\mathbf{v}) + \lambda\rho g, \quad \rho(\mathbf{x},0) = \mu, \ \rho(\mathbf{x},1) = \nu.
\end{align}
Here for any $t$, $\rho(\cdot,t)\in\mathcal{M}_+(X)$ is an unnormalized density, since the growth term does not guarantee the normalization $\int_X\dot{\rho}\ \text{d}\mathbf{x} = 0$. $\mathcal{M}_+(X)$ represents the set of all positive Radon measures on $X$. Unnormalized dynamical OT has been proposed \cite{gangbo2019unnormalized,lee2021generalized}, by taking the growth term in (\ref{eqn:continuity_growthterm}) as a more general form $g$ instead of $\rho g$, in time-dependent only form \cite{gangbo2019unnormalized} or in both time- and spatial-dependent form \cite{lee2021generalized}. A further extension of the dynamical OT has been developed to incorporate more general energy functionals by replacing the 2-Wasserstein norm of $\dot{\rho}$ in (\ref{eqn:DynamicalOT02}) with a general Lagrangian $\mathcal{L}(\rho,\dot{\rho},t)$ \cite{pooladian2023neural,neklyudov2023computational}.

\section{Synchronized Optimal Transport Formulation} \label{section:SynOT}

In this section, we propose {\it synchronized optimal transport} (SyncOT), to attain coherent dynamics across multiple spaces. Let $X^{(i)}, i=1,\cdots,d$ be $d$ spaces, with $X^{(i)} \subset \mathbb{R}^{N_i}$, SyncOT considers the following minimization problem
\begin{align}\label{eqn:SynOT_dSpaces}
 \inf_{(\rho^{(i)},\mathbf{v}^{(i)})} \sum_{i=1}^{d} \alpha_i \int_0^1 \int_{X^{(i)}} \rho^{(i)} \|\mathbf{v}^{(i)}\|^2  \ \text{d}\mathbf{x}^{(i)}\text{d}t
\end{align}
subject to the following constraints:
\begin{align}\label{eqn:SynOT_dSpaces_Constraints}
  &\dot{\rho}^{(i)} = -\nabla\cdot(\rho^{(i)}\mathbf{v}^{(i)}), \quad \rho^{(i)}(\mathbf{x},0) = \rho^{(i)}_0, \ \rho^{(i)}(\mathbf{x},1) = \rho^{(i)}_1, \ i=1,\cdots, d,\\
  &\rho^{(i)}(\bx^{(i)}, t) = \mathcal{T}^{(i)}\left(\rho^{(1)}(\bx^{(1)}, t) \right), \quad i=2,\ldots d.
\end{align}
Here $\alpha_i\ge 0$ and $\sum_{i} \alpha_i = 1$. The density $\rho^{(1)}(\bx,t)$ represents an evolving distribution in $X^{(1)}$ which we call the primary space . The dynamics in the primary space induces dynamics in other related spaces through maps $\mathcal{T}^{(i)}: \mathcal{M}_+(X^{(1)}) \rightarrow \mathcal{M}_+(X^{(i)})$. Note that the constraints in the SyncOT framework (\ref{eqn:SynOT_dSpaces}) naturally implicit a {\it marginal compatibility condition}: for each space $X^{(i)}: i=2\cdots, d$, there exists a map $\mathcal{T}^{(i)}$ such that
\begin{align}
  \rho_0^{(i)}(\bx^{(i)}) = \mathcal{T}^{(i)}\left(\rho_0^{(1)}(\bx^{(1)})\right), \quad \rho_1^{(i)}(\bx^{(i)}) = \mathcal{T}^{(i)}\left(\rho_1^{(1)}(\bx^{(1)})\right), \quad i=2,\cdots, d.
\end{align}

\begin{figure}[t]

  \centering

  \begin{tikzpicture}

    \draw [line width=0.5pt, xshift=0cm] plot [smooth, tension=0.7] coordinates { (-4,0) (-2,0.5) (0,0) (2,-0.5) (4,0)};

    \draw [line width=0.5pt, xshift=0cm, yshift=-2.5cm] plot [smooth, tension=0.5] coordinates { (-4,0) (-2.5,-0.3) (-1.0,-0.3) (0,0) (1.7,1)  (3.3,-0.1) (4,0.2)};

    \draw (-4,0) circle(2pt) [fill];
    \draw ( 4,0) circle(2pt) [fill];
    \draw (-4,-2.5) circle(2pt) [fill];
    \draw (4,-2.3) circle(2pt) [fill];

    \node at (-5.7,-2.5) {$\mathcal{T}\mu = \rho^{(2)}(\mathbf{x},0)$};
    \node at ( 5.7,-2.5) {$\rho^{(2)}(\mathbf{x},1) = \mathcal{T}\nu$};
    \node at (0, -3.0) {$\rho^{(2)}(\mathbf{x},t)$};
    \node at (-5.5, 0) {$\rho^{(1)}(\mathbf{x},0) = \mu$};
    \node at ( 5.5, 0) {$ \nu = \rho^{(1)}(\mathbf{x},1) $};
    \node at (0, -0.5) {$\rho^{(1)}(\mathbf{x},t)$};
    \node at (-2.0, -1.1) {$\mathcal{T}$};
    \draw[-{Triangle[width=5pt,length=8pt]}, line width=1pt](-0.5,0.5)--(0.5,0.14);
    \draw[-{Triangle[width=5pt,length=8pt]}, line width=1pt](-0.5,-2.4)--(0.5,-1.9);
    \draw[-{Triangle[width=5pt,length=8pt]}, line width=1pt](-2.5,-0.3)--(-2.5,-2.0);
              
  \end{tikzpicture}

  \caption{Schematic of the synchronized optimal transport. The two dynamics $\rho^{(1)}(\mathbf{x},t)$ and$\rho^{(2)}(\mathbf{x},t)$ are related by a transform $\mathcal{T}$. SyncOT seeks an optimal dynamics $\rho^{(1),*}$ in the primary space such that the combined kinetic energy in primary and secondary spaces is minimized.   }
\label{fig:Hybrid} 

\end{figure}
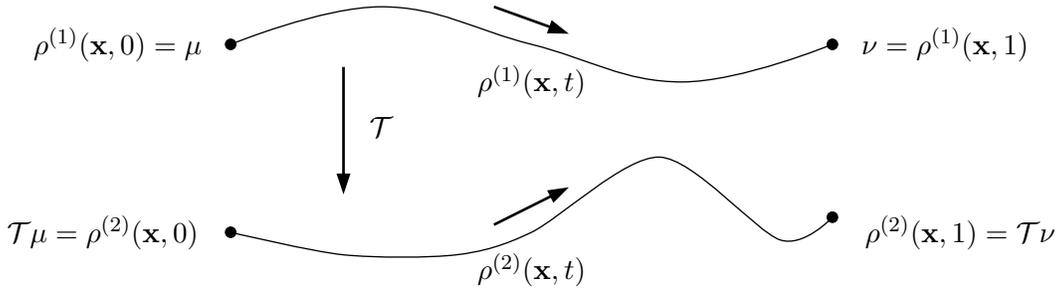

In summary, the marginal compatibility condition guarantees the wellposedness of the SyncOT framework. Figure \ref{fig:Hybrid} shows a schematic illustration of the SyncOT framework for $d=2$. 
Figure \ref{fig:1d_example} shows a one-dimensional example in which an optimal path in the primary space leads to a high-cost path in the secondary space. SyncOT addresses the secondary space by adjusting the optimality in the primary space to instead minimize the total costs across both spaces, thus achieving a more balanced dynamics. In this scenario, the path in the secondary space is almost constant speed due to the high weight assigned to the secondary space.

\begin{figure}[t]
\centering
\includegraphics[width=0.9\textwidth]{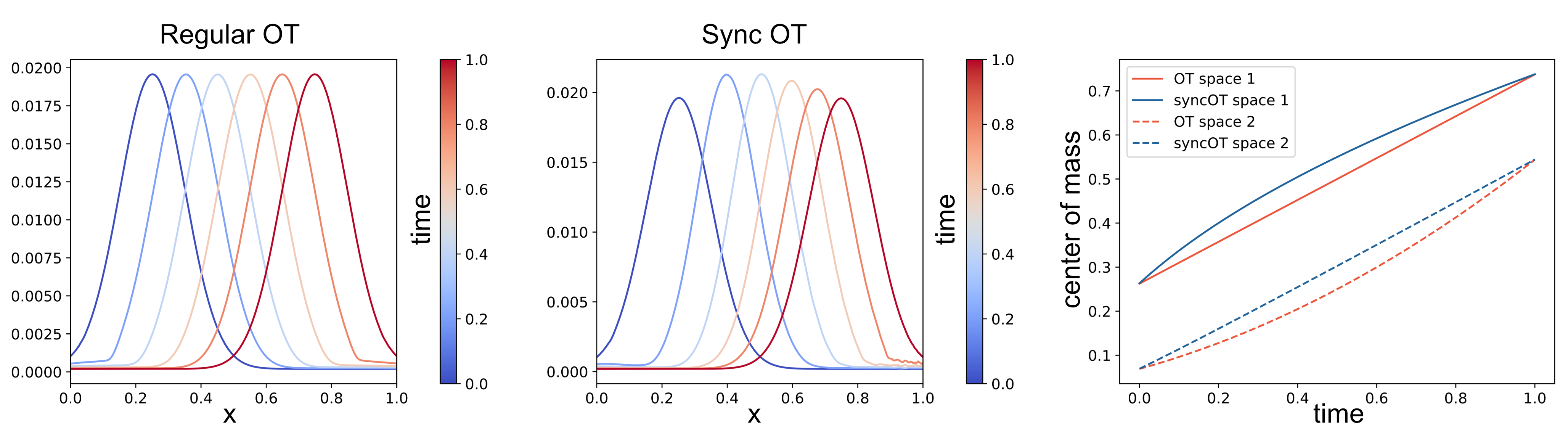}
\caption{A 1D example of SyncOT. The dynamics in  $X^{(1)}$ are shown for regular OT (left) and SyncOT (middle) with $\alpha_1=0.1,\alpha_2=0.9$. Regular OT achieves a constant-speed trajectory in $X^{(1)}$ but not in $X^{(2)}$. In contrary, SyncOT achieves a nearly constant-speed trajectory in space $x^{(2)}$ due to the large weight assigned to $X^{(2)}$. The two spaces are connected through the map $\mathbf{T}(x)=x^2$.}\label{fig:1d_example}
\end{figure}

In this work, we consider predefined maps $\mathcal{T}^{(i)}$ based on known correspondence between $X^{(1)}$ and $X^{(i)}$, such as known couplings between the datasets representing the supporting spaces. Specifically, we consider two commonly encountered forms of $\mathcal{T}^{(i)}$:
\vspace{-.1in}
\begin{itemize}
\item {\it Monge form}: $\mathcal{T}^{(i)} = \mathbf{T}_{\sharp}^{(i)}$ for a known map $\mathbf{T}^{(i)}:X^{(1)}\rightarrow X^{(i)}$. For instance, the map can be obtained by a trained neural network, then one can use the push-forward operator to define $\rho^{(i)} = \mathbf{T}^{(i)}_\sharp\rho^{(1)}$.
\vspace{-.1in}
\item {\it Kantorovich form}: $\mathcal{T}^{(i)} = \int_{\Omega^{(i)}}\text{d}\pi$ for a known joint distribution $\pi\in\mathcal{M}_+(X^{(1)}\times X^{(i)})$. For instance, $\pi$ can be obtained by solving Kantorovich OT on the two datasets representing $X^{(1)}$ and $X^{(i)}$, then one can define $\rho^{(i)}=\int_{\Omega^{(1)}}\rho^{(1)} \mathrm{d}\pi $.
\end{itemize}
\vspace{-.1in}

For the sake of brevity, we focus on the two-space case $d=2$ from now on. The extension of the theories and numerical methods for SyncOT to the general $d$-space case is straightforward. In what follows, we will consider the two-space SyncOT in both Monge and Kantorovich forms.

\subsection{Synchronized Optimal Transport in Monge Form}\label{subsection:SyncOT_Monge}

With a slight abuse of notation, we consider the Monge SyncOT as follows:
\begin{equation}\label{eqn:SynOT_2Spaces}
\begin{aligned}
  &C_{\text{M}}(\mu,\nu; \boldsymbol{\upalpha}) := \inf_{(\rho,\mathbf{u};\xi,\bv)} \int_0^1 \left( \alpha_1 \int_{X} \rho \|\mathbf{u}\|^2 \ \text{d}\mathbf{x}   + \alpha_2 \int_{Y} \xi \|\mathbf{v}\|^2 \ \text{d}\mathbf{y}  \right)\text{d}t, \\
  \text{s.t.} \quad  &\dot{\rho} = -\nabla\cdot(\rho\mathbf{u}), \quad \rho(\mathbf{x},0) = \mu, \ \rho(\mathbf{x},1) = \nu, \quad \bu\cdot\mathbf{n}|_{\partial X} = 0, \\
  & \dot{\xi} = -\nabla\cdot(\xi \mathbf{v}), \quad \xi(\cdot, t) = \mathbf{T}_{\sharp}\rho(\cdot, t),\quad \bv\cdot\mathbf{n}|_{\partial Y} = 0.
\end{aligned}
\end{equation}
where $\boldsymbol\upalpha = (\alpha_1,\alpha_2)\in\Sigma_2$. The density $\rho(\mathbf{x},t)$ represents an evolving distribution in the primary space $X$. The dynamics $\rho(\mathbf{x},t)$ in the primary space induces a dynamics $\xi(\bx,t)$ in the secondary space $Y$ through map $\mathbf{T}_{\sharp}:\mathcal{M}_+(X)\rightarrow\mathcal{M}_+(Y)$, with the assumption of marginal compatibility condition that $\mathbf{T}$ is a diffeomorphism such that 
\begin{align}
    \xi(\cdot,0) = \mathbf{T}_{\sharp}\mu, \quad 
    \xi(\cdot,1) = \mathbf{T}_{\sharp}\nu.
\end{align}
The condition of $\mathbf{T}$ being diffeomorphism is necessary to map the boundary of $X$ to the boundary of $Y$ such that the boundary conditions of $\bu$ and $\bv$ are both satisfied.

To numerically solve the SyncOT in Monge form (\ref{eqn:SynOT_2Spaces}), we recast the problem into a single-space dynamical OT form where the ground metric depends on $\mathbf{T}$. We summarize the equivalence as a theorem.

\begin{theorem}
The synchronized optimal transport (\ref{eqn:SynOT_2Spaces}) is equivalent to the following generalized dynamical OT:
\begin{equation}
  \begin{aligned}\label{eqn:MongeSyncOT_02}
      &\min_{(\rho,\mathbf{u})} \ \int_0^1 \int_{X} \rho\left\|\mathbf{u}\right\|^2_{A}\ \emph{d}\mathbf{x}\emph{d}t \\
      \emph{s.t.} \quad  &\dot{\rho} = -\nabla\cdot(\rho\mathbf{u}), \quad \rho(\mathbf{x},0) = \mu, \ \rho(\mathbf{x},1) = \nu, \quad \bu\cdot\mathbf{n}|_{\partial X} = 0,
  \end{aligned}
\end{equation}
where $A = \alpha_1 I+\alpha_2(\nabla \mathbf{T})^T(\nabla \mathbf{T})$, and $\|\mathbf{u}\|_{A}^2 := \mathbf{u}^T A\mathbf{u}$.
\end{theorem}

\begin{proof}
    To prove the equivalence, the key step is to calculate the kinetic energy in the secondary space. Indeed,
    \begin{align}\label{eqn:Tpush_eqn}
       \int_{Y} \xi (\mathbf{y},t) \|\bv(\by,t)\|^2 \text{d}\mathbf{y}
        = \int_{X} \rho(\bx,t) \|\bv(\mathbf{T}(\mathbf{x}),t)\|^2 \text{d}\mathbf{x}.
    \end{align}
    We replace $\bv$ in terms of $\bu$. Note that $\bv$ satisfies the transport equation
    \[
    \frac{\partial}{\partial t} \xi(\mathbf{y},t) + \nabla\cdot\left(\xi(\mathbf{y},t)\bv(\mathbf{y},t)\right) = 0,
    \]
    in the distributional sense, hence for any test function $\phi\in C_{c}^{\infty}(Y)$, it holds that
    \[
    \int_{Y} \frac{\partial}{\partial t} \xi(\mathbf{y},t) \phi(\mathbf{y}) \text{d}\mathbf{y}
    = - \int_{Y} \nabla\cdot\left(\xi(\mathbf{y},t)\bv(\mathbf{y},t)\right)\phi(\mathbf{y}) \text{d}\mathbf{y}.
    \]
    From the above equation, we have
    \begin{align*}
      \text{LHS} &= \frac{\text{d}}{\text{d}t} \int_{Y} \xi(\mathbf{y},t) \phi(\mathbf{y})  \text{d}\mathbf{y}  = \frac{\text{d}}{\text{d}t} \int_{X} \rho(\mathbf{x},t) \phi(\bT(\mathbf{x})) \text{d}\mathbf{x} \\
      & = \int_{X} \frac{\partial}{\partial t} \rho(\mathbf{x},t) \phi(\bT(\mathbf{x}))  \text{d}\mathbf{x} = -\int_{X} \nabla\cdot \left(\rho(\mathbf{x},t)\bu(\mathbf{x},t)\right)\phi(\bT(\mathbf{x})) \text{d}\mathbf{x} \\
      & = \int_{X}\nabla(\phi(\bT(\mathbf{x})))\cdot \left(\rho(\mathbf{x},t)\bu(\mathbf{x},t)\right)\text{d}\mathbf{x} = \int_{X} \Big[(\nabla\phi)(\bT(\mathbf{x}))\Big]^T (\nabla \bT) \bu(\mathbf{x},t)\rho(\mathbf{x},t)\text{d}\mathbf{x}, \\
    \intertext{\text{and}}
      \text{RHS} & = \int_{Y} \nabla\phi(\mathbf{y})\cdot \left(\xi(\mathbf{y},t)\bv(\mathbf{y},t)\right)\text{d}\mathbf{y} = \int_{Y} [\nabla\phi(\mathbf{y})]^T \bv(\mathbf{y},t)  \xi(\mathbf{y},t) \text{d}\mathbf{y} \\
      & = \int_{X} \Big[(\nabla\phi)(\bT(\mathbf{x}))\Big]^T \bv(\bT(\mathbf{x}),t) \rho(\mathbf{x},t) \text{d}\mathbf{x}.
    \end{align*}
    Since $\phi$ is arbitrary, it holds that
    \begin{align}\label{eqn:v_relation}
        \bv(\bT(\mathbf{x}),t) = (\nabla \bT) \bu(\mathbf{x},t).
    \end{align}
    Finally, inserting the relation (\ref{eqn:v_relation}) between $\bv$ and $\bu$ into (\ref{eqn:Tpush_eqn}), we have
    \begin{align*}
      \int_{Y} \xi (\mathbf{y},t) \|\bv(\by,t)\|^2 \text{d}\mathbf{y}
      = \int_{X} \rho(\bx,t) \|(\nabla \bT) \bu(\mathbf{x},t)\|^2 \text{d}\mathbf{x}
        &= \int_{X} \rho(\bx,t) \big\| \bu(\mathbf{x},t) \big\|^2_{(\nabla \bT)^T(\nabla \bT)} \text{d}\mathbf{x}.
    \end{align*}
    The proof is therefore completed.
\end{proof}

\subsection{Synchronized Optimal Transport in Kantorovich Form}\label{subsection:SyncOT_Kantorovich}

We consider the Kantorovich SyncOT as follows:
\begin{equation}\label{eqn:KantorovichSynOT}
  \begin{aligned}
    &C_{\text{K}}(\mu,\nu; \boldsymbol{\upalpha}) := \inf_{(\rho,\mathbf{u};\bv)} \int_0^1 \left( \alpha_1 \int_{X} \rho \|\mathbf{u}\|^2 \ \text{d}\mathbf{x}   + \alpha_2 \int_{Y} \xi \|\mathbf{v}\|^2 \ \text{d}\mathbf{y}  \right)\text{d}t, \\
    \text{s.t.} \quad  &\dot{\rho} = -\nabla\cdot(\rho\mathbf{u}), \quad \rho(\mathbf{x},0) = \mu, \ \rho(\mathbf{x},1) = \nu, \quad \bu\cdot\mathbf{n}|_{\partial X} = 0; \\
    & \dot{\xi} = -\nabla\cdot(\xi\mathbf{v}), \quad \xi(\mathbf{y}, t) = \int_{X} \rho(\mathbf{x}, t) \text{d}\pi, \quad \bv\cdot\mathbf{n}|_{\partial Y} = 0.
  \end{aligned}
\end{equation}
Here $\pi\in\mathcal{M}_+(X \times Y)$ is the given joint distribution between $X$ and $Y$. Note that the Kantorovich SyncOT minimizes the total kinetic energy over $(\rho,\mathbf{u};\mathbf{v})$ but not $\xi$. This is because $\xi$ is determined by $\rho$ through the Kantorovich map $\xi = \int_X \rho \text{d}\pi$. Similar to the Monge case, we assume the marginal compatibility condition that the joint distribution $\pi \ll \text{d}x\text{d}y$ satisfies 
\begin{align}
    \xi(\mathbf{y},0) = \int_{X} \mu(\mathbf{x}) \text{d}\pi, \quad 
    \xi(\mathbf{y},1) = \int_{X} \nu(\mathbf{x}) \text{d}\pi.
\end{align}

Unlike the Monge SyncOT which has an equivalent form (\ref{eqn:MongeSyncOT_02}) that only depends on the primary space and the transform $\bT$, it is not straightforward to rewrite the dynamics in the secondary space in terms of the variables $(\rho, \bu)$ in the primary space. An alternative way is needed to reformulate the Kantorovich SyncOT (\ref{eqn:KantorovichSynOT}) for feasible numerical implementation. To this end, we introduce some mild conditions on the feasible density $\rho(\mathbf{x},t)$. Specifically, we assume that any feasible density $\rho(\mathbf{x},t)$ is absolutely continuous with respect to Lebesgue measure $\text{d}x$, namely, $\rho(t,\cdot) \ll \text{d}x$ for any $t\in[0,1]$. Under this assumption, the Monge OT (\ref{eqn:MongeOT}) and  Kantorovich OT (\ref{eqn:KantorovichOT}) between $\rho(t_1,\cdot)$ and $\rho(t_2,\cdot)$ are equivalent for $0\le t_1 \le t_2 \le 1$ by Brenier's theorem \cite{brenier1991polar}. In addition, when two distributions $\rho(t,\cdot)$ and $\rho(t+\text{d}t,\cdot)$ are infinitesimally close to each other, it holds that \cite{neklyudov2023action,gangbo1996geometry,ambrosio2005gradient,villani2009optimal}
\begin{align}
  &\ W_2^2(\rho(t,\cdot),\rho(t+\text{d}t,\cdot)) \\
  =&\ \inf_{\mathbf{T}} \left\{\int_X \|\mathbf{T}(\mathbf{x})-\mathbf{x}\|^2 \rho(t,\mathbf{x})\text{d}\mathbf{x}: T_{\sharp} \rho(t,\cdot) = \rho(t+\text{d}t,\cdot) \right\} \\
  =&\ \inf_{\mathbf{v}} \left\{ \int_X\int_{t}^{t+\text{d}t} \rho(t,\mathbf{x})\|\mathbf{v}\|^2 \text{d}\mathbf{x}\text{d}t: (\mathbf{x}+\text{d}t\mathbf{v})_{\sharp}\rho(t,\cdot) = \rho(t+\text{d}t,\cdot)   \right\}.
\end{align}
Note that the given joint distribution $\pi$ is absolutely continuous with respect to the Lebesgue measure, therefore we can use the Riemann sum approximation to reformulate the Kantorovich SyncOT as follows
\begin{equation}\label{eqn:KantorovichSynOT_02}
  \begin{aligned}
    &C_{\text{K}}(\mu,\nu; \boldsymbol{\upalpha}) \approx \inf_{(\rho,\mathbf{u})}\  \alpha_1 \int_0^1 \int_{X} \rho \|\mathbf{u}\|^2 \ \text{d}\mathbf{x}  + \alpha_2 \sum_{i=0}^{N-1} \frac{1}{\Delta t_i} W_2^2 \Big( \xi(\mathbf{y},t_i), \xi(\mathbf{y},t_{i+1}) \Big), \\
    \text{s.t.} \quad  &\dot{\rho} = -\nabla\cdot(\rho\mathbf{u}), \ \rho(\mathbf{x},0) = \mu, \ \rho(\mathbf{x},1) = \nu, \ \bu\cdot\mathbf{n}|_{\partial X} = 0, \ \xi(\mathbf{y}, t) = \int_{X} \rho(\mathbf{x}, t) \text{d}\pi,
  \end{aligned}
\end{equation}
where $0\le t_0 < t_1 < \cdots < t_{N-1} < t_N = 1$ is a Riemann sum partition for the time interval $[0,1]$ and  $\Delta t_i = t_{i+1} - t_{i}$.



\section{Numerical Methods for Synchronized Optimal Transport in Monge Form}\label{section:Numerical_Monge} 

In this section, we present numerical methods for solving the SyncOT problem in Monge form. For simplicity, we focus on the 2-dimensional case and take the primary space $X=[0,1]^2$. The secondary space $Y$ depends on the transform $\bT$.

To begin with, we follow \cite{benamou2000computational} to introduce the change of variable $(\bu,\rho) \rightarrow (\bfm,\rho)$ where $\bfm$ is the momentum $\bfm = \rho \bu$. Then the SyncOT problem can be recast into a convex optimization problem over the pair $(\bfm,\rho)$:
\begin{align}\label{eqn:SyncOT_momentum}
\min_{(\bfm,\rho)\in\calC} \mathcal{J}(\bfm,\rho) = \int_0^1 \int_{X} J(\bfm(\bx,t),\rho(\bx,t)) \ \text{d}\bx\text{d}t
\end{align}
where the set of constraints is defined as
\begin{align}\label{eqn:SyncOT_momentum_constraint}
  \mathcal{C}(\mu,\nu): = \left\{ (\bfm,\rho): \dot{\rho} + \nabla\cdot\bfm = 0, \ \rho(\bx,0)=\mu, \rho(\bx,1)=\nu \right\},
\end{align}
and $J$ is the following lower semi-continuous convex function
\begin{align}\label{eqn:J_definition}
  J(\bfm,\rho) = 
  \begin{cases}
    \rho^{-1}\|\bfm\|_A^2 \quad &\text{if\ } \rho >0, \\
    0 \quad &\text{if\ } (\bfm,\rho) = (0,0), \\
    +\infty, \quad & \text{otherwise}.
  \end{cases}
\end{align}
for a positive-definite matrix $A = \alpha_1 I + \alpha_2(\nabla \mathbf{T})^T(\nabla \mathbf{T})$.

\subsection{Discretization on Uniform Staggered Grids}\label{subsection:Staggered}

In this section, we introduce an uniform staggered grid for the discretization of the SyncOT (\ref{eqn:SyncOT_momentum}). Uniform staggered grids are commonly used in fluid dynamics \cite{harlow1965numerical}. See a schematic illustration in Figure \ref{fig:my_label}. To begin with, we denote by $M+1, N+1, Q+1$ the numbers of uniform staggered nodes along $\bx = (x, y)\in[0,1]^2$ and time $t\in[0,1]$. The corresponding grid sizes are $\Delta x = \frac{1}{M}, \Delta{y} = \frac{1}{N}, \Delta{t} = \frac{1}{Q}$.

The following index sets of staggered grids are defined:
\begin{align}
  \mathcal{S}_{\text{s}}^{x} &= \{ (i,j,k):\ 0\le i \le M, \ 0\le j \le N-1, \ 0\le k \le Q-1 \}, \\
  \mathcal{S}_{\text{s}}^{y} &= \{ (i,j,k):\ 0\le i \le M-1, \ 0\le j \le N, \ 0\le k \le Q-1 \}, \\ 
  \mathcal{S}_{\text{s}}^{t} &= \{ (i,j,k):\ 0\le i \le M-1, \ 0\le j \le N-1, \ 0\le k \le Q \}.
\end{align}
The staggered nodes over $(\bx,t) \in [0,1]^2\times [0,1]$  are given as
\begin{align}
  \mathcal{G}_{\text{s}}^x &= \left\{  (x_{i-\frac{1}{2}}, y_j, t_k) = \left( i\Delta x, \left(j + 1/2 \right)\Delta y, \left(k + 1/2 \right)\Delta t \right), \ (i,j,k) \in \mathcal{S}_{\text{s}}^{x}   \right\}, \\
  \mathcal{G}_{\text{s}}^y &= \left\{  (x_{i}, y_{j-\frac{1}{2}}, t_k) = \left( \left(i + 1/2\right)\Delta x, j\Delta y, \left(k + 1/2 \right)\Delta t \right), \ (i,j,k) \in \mathcal{S}_{\text{s}}^{y}    \right\}, \\
  \mathcal{G}_{\text{s}}^t &= \left\{  (x_{i}, y_{j}, t_{k-\frac{1}{2}}) = \left( \left(i + 1/2\right)\Delta x, \left(j + 1/2\right)\Delta y, k \Delta t \right), \ (i,j,k) \in \mathcal{S}_{\text{s}}^{t}   \right\}.
\end{align}
The discrete functions defined on the staggered grid $\mathcal{G}_{\text{s}}^{q}, q = x,y,t$ are denoted by $\mathcal{M}(\mathcal{G}_{\text{s}}^{q})$, namely, for $f^{q}\in \mathcal{M}(\mathcal{G}_{\text{s}}^{q}), q = x, y, t$,
\begin{align}
  f^x = \left( f^x_{i-\frac{1}{2},j,k} \right)_{(i,j,k) \in \mathcal{S}_{\text{s}}^{x} }, \quad 
  f^y = \left( f^y_{i,j-\frac{1}{2},k} \right)_{(i,j,k) \in \mathcal{S}_{\text{s}}^{y} }, \quad
  f^t = \left( f^t_{i,j,k-\frac{1}{2}} \right)_{(i,j,k) \in \mathcal{S}_{\text{s}}^{t} }.
\end{align}
We further denote the discrete functions defined on $\mathcal{G}_{\text{s}}^{x} \times \mathcal{G}_{\text{s}}^{y} \times \mathcal{G}_{\text{s}}^{t}$ as
\begin{align*}
  U_{\text{s}} = (\bfm_{\text{s}},\rho_{\text{s}}) = (m_{\text{s}}, n_{\text{s}},\rho_{\text{s}}) = (m_{i-\frac{1}{2},j,k}, n_{i,j-\frac{1}{2},k}, \rho_{i,j,k-\frac{1}{2}})\in \mathcal{D}_{\text{s}}= \mathcal{M}(\mathcal{G}_{\text{s}}^{x}) \times \mathcal{M}(\mathcal{G}_{\text{s}}^{y}) \times \mathcal{M}(\mathcal{G}_{\text{s}}^{t}).
\end{align*}





  



\begin{figure}[t]
  \centering
  \includegraphics[width=0.4\textwidth]{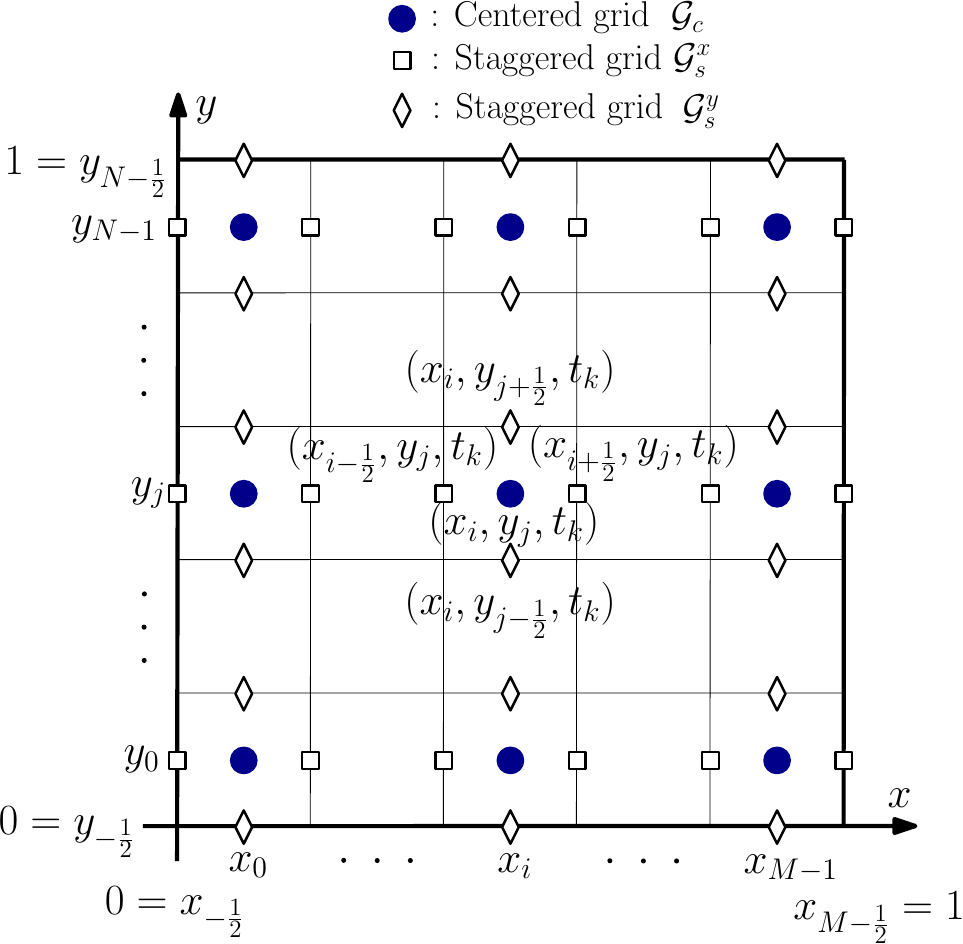} \qquad
  \includegraphics[width=0.4\textwidth]{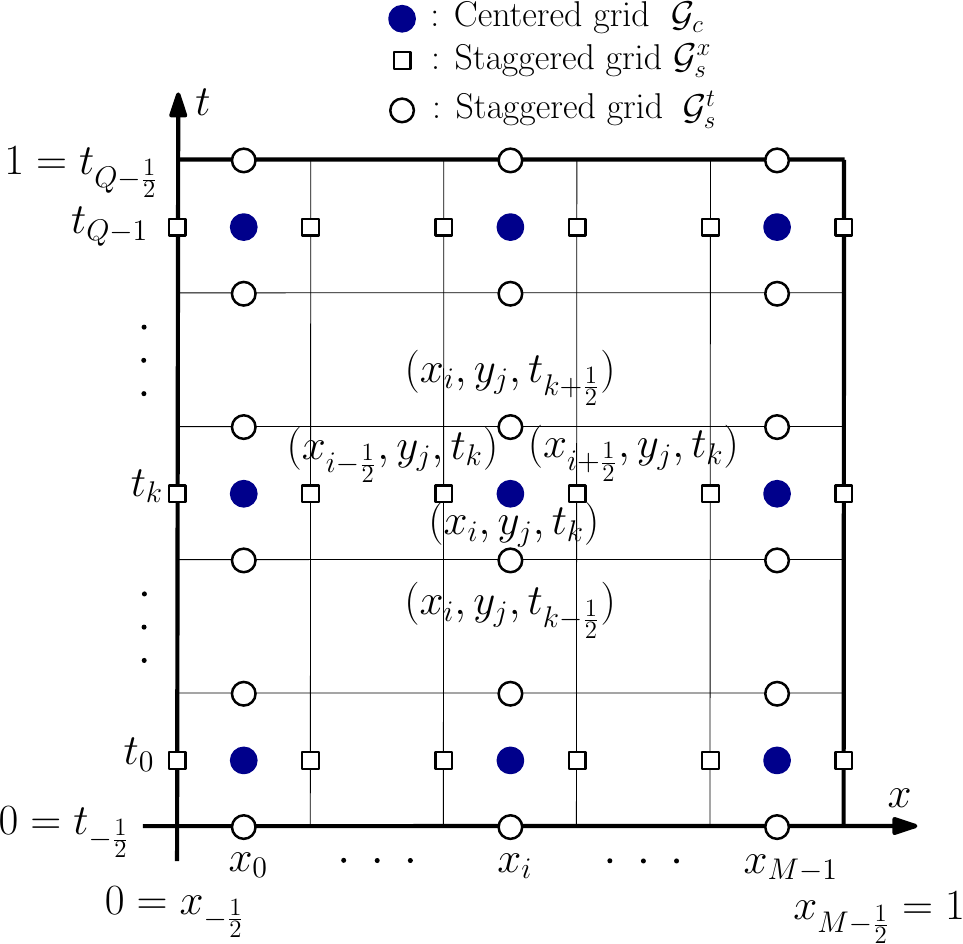}
  \caption{Schematic of the staggered grid. Left: centered and staggered grids on the $t_k$-slice; Right: centered and staggered grids on the $y_j$-slice.}
  \label{fig:my_label}
\end{figure}

We also define the index set of centered grid as
\begin{align}
  \mathcal{S}_{\text{c}} = \{ (i,j,k)\in\mathbb{Z}^3:\  0\le i \le M-1,\  0\le j \le N-1,\  0\le k \le Q-1 \}.
\end{align}
Then a centered grid discretization over $(\bx,t) \in [0,1]^2\times [0,1]$ is given as
\begin{align}
  \mathcal{G}_{\text{c}} = \left\{  (x_i, y_j, t_k) = \left( (i+1/2) \Delta x, (j+1/2)\Delta y, (k+1/2)\Delta t \right),\  (i,j,k) \in \mathcal{S}_{\text{c}}  \right\}.
\end{align}
The discrete functions defined on the centered grid $\mathcal{G}_{\text{c}}$ are denoted by $\mathcal{M}(\mathcal{G}_{\text{c}})$. The SyncOT variable $(\bfm,\rho)$ discretized over $\mathcal{G}_{\text{c}} \times \mathcal{G}_{\text{c}} \times \mathcal{G}_{\text{c}}$ is denoted by
\begin{align}
  U_{\text{c}} = (\bfm_{\text{c}},\rho_{\text{c}})  = (m_{\text{c}}, n_{\text{c}},\rho_{\text{c}}) = (m_{ijk}, n_{ijk}, \rho_{ijk}) \in \mathcal{D}_{\text{c}} = \mathcal{M}(\mathcal{G}_{\text{c}}) \times \mathcal{M}(\mathcal{G}_{\text{c}}) \times \mathcal{M}(\mathcal{G}_{\text{c}}).
\end{align}

We further introduce a midpoint interpolation operator $\mathcal{I}: \mathcal{D}_{\text{s}} \rightarrow \mathcal{D}_{\text{c}}$. For $U_{\text{s}} = (\bfm_{\text{s}},\rho_{\text{s}}) \in \mathcal{D}_{\text{s}}$, we define $\mathcal{I}(U_{\text{s}}) = (\bfm_{\text{c}},\rho_{\text{c}}) \in \mathcal{D}_{\text{c}}$ as follows:
\begin{align}
  \begin{cases}
  m_{ijk} = \frac{1}{2}( m_{i-\frac{1}{2},j,k} + m_{i+\frac{1}{2},j,k}), \\ 
  n_{ijk} = \frac{1}{2}( n_{i,j-\frac{1}{2},k} + n_{i,j+\frac{1}{2},k}), \\
  \rho_{ijk} = \frac{1}{2}( \rho_{i,j,k-\frac{1}{2}} + \rho_{i,j,k+\frac{1}{2}}),
  \end{cases}
  \quad (i,j,k)\in\mathcal{S}_{\text{c}}.
\end{align}
Then the space-time divergence operator $\div: \mathcal{D}_{\text{s}} \rightarrow \mathcal{M}(\mathcal{G}_{\text{c}}) $ is defined for $U_{\text{s}} = (\bfm_{\text{s}},\rho_{\text{s}}) \in \mathcal{D}_{\text{s}}$:
\begin{align}
  \div(U_{\text{s}})_{ijk} =  \frac{m_{i+\frac{1}{2},j,k} - m_{i-\frac{1}{2},j,k}}{\Delta x}  + 
                              \frac{ n_{i,j+\frac{1}{2},k} - n_{i, j-\frac{1}{2},k} }{\Delta y}  +
                               \frac{\rho_{i,j,k+\frac{1}{2}} - \rho_{i,j,k-\frac{1}{2}} }{\Delta t},
\end{align}
for $(i,j,k)\in\mathcal{S}_{\text{c}}$.

To handle the boundary conditions on the staggered grids, we introduce the linear operator $b$ for $U_{\text{s}} = (\bfm_{\text{s}},\rho_{\text{s}}) \in \mathcal{D}_{\text{s}}$ as
\begin{align*}
  b(U_{\text{s}}) = \left( \left(m_{-\frac{1}{2},j,k}, m_{M-\frac{1}{2},j,k}\right)_{j=0:N-1}^{k=0:Q-1}, 
                          \left(n_{i,-\frac{1}{2},k}, n_{i,N-\frac{1}{2},k}\right)_{i=0:M-1}^{k=0:Q-1},
                          \left( \rho_{i,j,-\frac{1}{2}}, \rho_{i,j,Q-\frac{1}{2}}\right)_{i=0:M-1}^{j=0:N-1}
                          \right),
\end{align*}
and impose the boundary conditions
\begin{align*}
  b(U_{\text{s}}) = b_0 : = ((0,0),(0,0),(\mu_{\text{c}}, \nu_{\text{c}}))\in (\mathbb{R}^{NQ})^2 \times (\mathbb{R}^{MQ})^2 \times (\mathbb{R}^{MN})^2,
\end{align*}
where $\mu_{\text{c}}, \nu_{\text{c}} \in \mathbb{R}^{MN}$ are the discretized initial and terminal densities over centered grid $\{ (x_i,y_j) = ( (i+1/2)\Delta x, (j+1/2)\Delta y ),\ i=0:M-1, j=0:N-1 \}$.

Finally, the SyncOT problem  (\ref{eqn:SyncOT_momentum}) is approximated by solving the following finite-dimensional convex problem over the staggered grid:
\begin{align}\label{eqn:SyncOT_discretized}
  \min_{U_{\text{s}}\in \mathcal{D}_{\text{s}}} \mathcal{J}(\mathcal{I}(U_{\text{s}})) + \iota_{\mathcal{C}}(U_{\text{s}}),
\end{align}
where, with a slight abuse of notation, we take
\begin{align}\label{eqn:J_discretized}
  \mathcal{J}(U_{\text{c}}) = \sum_{(i,j,k)\in\mathcal{S}_{\text{c}}} J(m_{ijk}, n_{ijk}, \rho_{ijk} ), \quad U_{\text{c}} = (\bfm_{\text{c}},\rho_{\text{c}}) \in \mathcal{D}_{\text{c}},
\end{align}
and $\iota_{\mathcal{C}}(U_{\text{s}})$ is the indicator function
\begin{align*}
  \iota_{\mathcal{C}}(U_{\text{s}}) = 
  \begin{cases}
     0, \ & \text{if\ } U_{\text{s}} \in \mathcal{C}, \\
     +\infty, \ & \text{otherwise}.
  \end{cases}
\end{align*}
with $\mathcal{C}$ being defined as
\begin{align}\label{eqn:C}
  \mathcal{C} = \{ U_{\text{s}}\in\mathcal{D}_s: \ \div(U_{\text{s}}) = 0,\  b(U_{\text{s}}) = b_0 \}.
\end{align}

\subsection{Primal-dual Methods for Synchronized Optimal Transport }\label{subsection:Primal_dual}

In this section, we adopt a primal-dual method introduced in \cite{chambolle2011first} which aims to minimize functionals of the form $f\circ K(x) + g(x)$. Here $K$ is a linear operator. $f, g$ are proximal friendly, namely, their corresponding proximal operators are in closed form, or can be effectively computed. We take
\[
x = U_{\text{s}}\in \mathcal{D}_{\text{s}}, \quad f = \mathcal{J}, \quad K = \mathcal{I}, \quad g = \iota_{\mathcal{C}},  
\]
to apply the primal-dual method to the discretized SyncOT problem (\ref{eqn:SyncOT_discretized}).

As an algorithm based on proximal splitting, primal-dual methods require computing the proximal operators for $\mathcal{J}$ and $\iota_{\mathcal{C}}$. In what follows, we show that the cost functional $\mathcal{J}$ in (\ref{eqn:J_discretized}) and the indicator function $\iota_{\mathcal{C}}$ are proximal friendly.

\begin{lemma}[$\text{prox}_{\tau\mathcal{J}}$] \label{lemma:proxJ}
  For the seperable discrete cost functional $\mathcal{J}$ defined in (\ref{eqn:J_discretized}) with $J$ given in (\ref{eqn:J_definition}), the proximal operator $\emph{prox}_{\tau\mathcal{J}}$ can be computed component-wisely,
  \begin{align}
    \emph{prox}_{\tau\mathcal{J}}(U_{\emph{c}}) = \Big\{\emph{prox}_{\tau J}(m_{ijk}, n_{ijk}, \rho_{ijk}) \Big\}_{(i,j,k)\in\mathcal{S}_{\emph{c}}}.
  \end{align}
  Here for any $(\bfm,\rho)\in \mathbb{R}^2\times\mathbb{R}$,
  \begin{align}
    \emph{prox}_{\tau J}(\bfm,\rho) = 
    \begin{cases}
      \big(\bfm^*(\rho^*), \rho^* \big), \quad &\emph{if\ } \rho>0, \\
      (0,0),  \quad  & \emph{otherwise},
    \end{cases} 
  \end{align}
  where $\rho^*$ is the unique fixed point over $(0,\infty)$ for $\phi(\tilde{\rho})$:
  \begin{align}
    \phi(\tilde{\rho}) = \rho +  \tau \big\langle (2\tau A + \tilde{\rho} I)^{-1}\bfm, A (2\tau A + \tilde{\rho} I)^{-1}\bfm \big\rangle,
  \end{align}
  and 
  \begin{align}
    \bfm^*(\rho^*) = (2\tau A + \rho^* I )^{-1} \rho^*\bfm.
  \end{align}

\end{lemma}

\begin{proof}
   Let $\text{prox}_{\tau J}(\bfm, \rho) = (\bfm^*, \rho^*)$. By the definition of the proximal operator,
   \[
   \text{prox}_{\tau J} (\bfm,\rho) = \underset{(\tilde{\bfm}, \tilde{\rho})}{\mathrm{argmin}} J( \tilde{\bfm}, \tilde{\rho} ) + \frac{1}{2\tau}\| (\bfm,\rho) - (\tilde{\bfm},\tilde{\rho}) \|^2,
   \]
and the definition of $J$ in (\ref{eqn:J_definition}), it is evident that if $0\ge\rho\in\mathbb{R}$, then $\text{prox}_{\tau J} (\bfm,\rho) = (\mathbf{0},0)$. Besides, if $\bfm = \mathbf{0}$ and $\rho>0$, then $\text{prox}_{\tau J} (\mathbf{0},\rho) = (\mathbf{0},\rho)$.

On the other hand, if $\rho>0$ and $\bfm \neq \mathbf{0}$, taking the first order derivative of the objective function for the proximal operator $\text{prox}_{\tau J} (\bfm,\rho)$, it yields
\begin{align}\label{eqn:proxJ_system}
  \begin{cases}
    2\tau \tilde{\rho}^{-1}(A \tilde{\bfm}) + \tilde{\bfm} - \bfm = 0, \\
    -\tau \tilde{\rho}^{-2}(\tilde{\bfm}^T A \tilde{\bfm}) + \tilde{\rho} - \rho = 0.
  \end{cases}
\end{align} 
The first equation in (\ref{eqn:proxJ_system}) leads to
\begin{align}\label{eqn:m_tilde}
  \tilde{\bfm} = (2\tau A + \tilde{\rho} I )^{-1} \tilde{\rho}\bfm.
\end{align}
Inserting the above equation to the second equation in (\ref{eqn:proxJ_system}) gives
\begin{align}\label{eqn:rho_tilde}
\tilde{\rho} = \rho + \tau \frac{\langle \tilde{\bfm}, A\tilde{\bfm} \rangle}{\tilde{\rho}^2} = \rho + \tau \Big\langle (2\tau A + \tilde{\rho} I)^{-1}\bfm, A (2\tau A + \tilde{\rho} I)^{-1}\bfm \Big\rangle.
\end{align}
Let $\phi(\tilde{\rho}) = \rho + \tau \big\langle (2\tau A + \tilde{\rho} I)^{-1}\bfm, A (2\tau A + \tilde{\rho} I)^{-1}\bfm \big\rangle$, we assert that $\phi'(\tilde{\rho})<0$ for $\bfm\neq \mathbf{0}$. Indeed, by taking the derivative,
\begin{align}\label{eqn:phi_prime}
  \phi'(\tilde{\rho}) &= \tau \Big\langle 2A(2\tau A + \tilde{\rho}I)^{-1}\bfm, \ \frac{\text{d}}{\text{d}\tilde{\rho}}(2\tau A + \tilde{\rho}I)^{-1}\bfm \Big\rangle \\
  &= -\tau \Big\langle 2A(2\tau A + \tilde{\rho}I)^{-1}\bfm, \ (2\tau A + \tilde{\rho}I)^{-2}\bfm \Big\rangle,
\end{align}
and noting $A = \alpha_1 I + \alpha_2 \nabla\mathbf{T}^T\nabla\mathbf{T}$ is positive definite, it follows that $\phi'(\tilde{\rho}) < 0 $ for $\bfm \neq \mathbf{0}$. The left hand side of the equation (\ref{eqn:rho_tilde}) is a linearly increasing function over the interval $\tilde{\rho}\in (0,\infty)$, beginning its growth from $0$; in contrast, the right hand side of the equation (\ref{eqn:rho_tilde}) is a strictly decreasing function over the same interval, starting from $\phi(0)=\frac{1}{4\tau}\langle \bfm, A^{-1}\bfm \rangle > 0 $, therefore there exists unique solution $\tilde{\rho} = \rho^*$ over $(0,\infty)$. Inserting the unique $\rho^*$ back to (\ref{eqn:m_tilde}) leads to $\bfm^* = (2\tau A + \rho^* I )^{-1} \rho^*\bfm $.
\end{proof}

Note that Lemma \ref{lemma:proxJ} provides a natural fixed point iteration algorithm to compute $\rho^*$. We summarize it in the following lemma.
\begin{lemma}\label{lemma:iteration_proxJ}
Taking any initial guess $\tilde{\rho}^{(0)} \in (0, \infty)$, the following fixed point iteration for $\phi$
\begin{align}
  \tilde{\rho}^{(l+1)} = \phi(\tilde{\rho}^{(l)}) =  \rho + \tau \Big\langle \left(2\tau A + \tilde{\rho}^{(l)} I \right)^{-1}\bfm, A \left(2\tau A + \tilde{\rho}^{(l)} I \right)^{-1}\bfm \Big\rangle, \quad l=0,1,\cdots,
\end{align}
is convergent for sufficiently large $\tau$.

\end{lemma}

\begin{proof}
Note that $\phi(\cdot)$ is monotonically decreasing and strictly convex, so we have that $\phi'(0) < \phi'(\tilde{\rho}) < 0$ for any $\tilde{\rho}^{(0)} \in (0, \infty)$. On the other hand, since
\[
  \phi'(0) = -\frac{1}{4\tau^2}\bfm^TA^{-2}\bfm,
\]
one can take sufficiently large $\tau$ such that $\phi'(0)>-1$. Then the condition $-1 < \phi'(\tilde{\rho}) \le 0$ guarantees that the fixed point iteration is always convergent.
\end{proof}

The computation of the proximal operator for $\iota_{\mathcal{C}}$ follows the results in \cite{papadakis2014optimal}. Note that $\text{prox}_{\iota_\mathcal{C}}= \text{proj}_{\mathcal{C}}$ with $\mathcal{C}$ defined in (\ref{eqn:C}). Taking a similar notation as in \cite{papadakis2014optimal},
\[
  A U_{\text{s}} =(\text{div}(U_{\text{s}}), b(U_{\text{s}})), \quad y = (\mathbf{0}, b_0),
\]
then the convex set $\mathcal{C}$ can be recast into
\begin{align}\label{eqn:C02}
  \mathcal{C} = \{ U_{\text{s}} = (\bfm_{\text{s}},\rho_{\text{s}}) \in\mathcal{D}_s: \ A U_{\text{s}} = y \}.
\end{align}
The projection $\text{proj}_{\mathcal{C}}$ can be computed directly by the following formula
\[
  \text{proj}_{\mathcal{C}}( U_{\text{s}} ) = U_{\text{s}} - A^*(AA^*)^{-1}(AU_{\text{s}} -y). 
\]
We summarize it in the following lemma \cite{papadakis2014optimal}.

\begin{lemma}[$\text{prox}_{\iota_{\mathcal{C}}}$]\label{lemma:iteration_projC}
  The proximal operator $\emph{prox}_{\iota_{\mathcal{C}}} = \emph{proj}_{\mathcal{C}}$ can be computed by 
\begin{align}\label{eqn:proj_update}
  U_{\text{s}}^{(l+1)} = \emph{proj}_{\mathcal{C}}( U_{\text{s}}^{(l)} ) = U_{\text{s}}^{(l)} - A^*(AA^*)^{-1}(AU_{\text{s}}^{(l)} -y). 
\end{align}
Equivalently, the iteration can be performed by firstly solving the Poisson equation
\begin{align}
  - ( D_x^2 + D_y^2 + D_t^2 ) s_{ijk} = \emph{div}(U^{(l)}_{\text{s}}), \quad s_{ijk} \in \mathcal{M}(\mathcal{G}_{\text{s}}^{c})
\end{align}
with homogeneous Neumann boundary condition where $D_q^2, q=x,y,t$ represents the 2nd-order central difference along $q$-direction. This can be achieved by the fast discrete cosine transform in $O(MNQ\log(MNQ))$ operations. Then the update $U_{\text{s}}^{(l+1)} \leftarrow U_{\text{s}}^{(l)}$ is achieved by
\begin{align}
  m^{(l+1)}_{i-\frac{1}{2},j,k} &= m^{(l)}_{i-\frac{1}{2},j,k} + \frac{1}{\Delta x} (s_{i,j,k} - s_{i-1,j,k}), \ (i,j,k)\in \mathcal{S}_{\text{s}}^{x}, \ i \neq 0, M ; \\
  n^{(l+1)}_{i, j-\frac{1}{2},k} &= n^{(l)}_{i, j-\frac{1}{2},k} + \frac{1}{\Delta y} (s_{i,j,k} - s_{i,j-1,k}), \ (i,j,k)\in \mathcal{S}_{\text{s}}^{y}, \ j \neq 0, N ; \\
  \rho^{(l+1)}_{i, j, k-\frac{1}{2} } &= \ \rho^{(l)}_{i, j, k-\frac{1}{2} } + \frac{1}{\Delta t} (s_{i,j,k} - s_{i,j,k-1}), \ (i,j,k)\in \mathcal{S}_{\text{s}}^{t}, \ k \neq 0, Q .
\end{align}

\end{lemma}
\begin{proof}
The proof is straightforward by directly plugging in the explicit expression of $A$ and $A^*$ into the update (\ref{eqn:proj_update}). So we omit the details. 
\end{proof}

Now we are ready to present the Chambolle-Pock primal-dual algorithm. The Chambolle-Pock algorithm generates a sequence $(U_s^{(l)}, \hat{U}_s^{(l)}, U_c^{(l)} ) \in \mathcal{D}_\text{s} \times \mathcal{D}_\text{s} \times \mathcal{D}_\text{c} $, starting from initial $(U_s^{(0)}, \hat{U}_s^{(0)}, U_c^{(0)} ) \in \mathcal{D}_\text{s} \times \mathcal{D}_\text{s} \times \mathcal{D}_\text{c} $ by the following iterations
\begin{align}\label{eqn:PD_Iterations_2terms}
  \begin{cases}
     U_{\text{c}}^{(l+1)} = \text{prox}_{\tau\mathcal{J}^*} (U_{\text{c}}^{(l)} + \tau \mathcal{I}(U_{\text{s}}^{(l)}) ), \\
     \hat{U}_{\text{s}}^{(l+1)} = \text{proj}_{\mathcal{C}} (\hat{U}_{\text{s}}^{(l)} - \sigma \mathcal{I}^*(U_{\text{c}}^{(l+1)}) ), \\
     U_{\text{s}}^{(l+1)} = \hat{U}_{\text{s}}^{(l+1)} + \theta (\hat{U}_{\text{s}}^{(l+1)} - \hat{U}_{\text{s}}^{(l)}).
  \end{cases}
\end{align}
Here $\text{prox}_{\tau\mathcal{J}^*}$ is calculated as follows
\begin{align}\label{eqn:Moreau}
  \text{prox}_{\tau\mathcal{J}^*}(U_{\text{c}}) = U_{\text{c}} - \tau \text{prox}_{\tau^{-1}\mathcal{J}}(\tau^{-1}U_{\text{c}}).
\end{align}
The operators $\text{prox}_{\tau\mathcal{J}}$ and $\text{proj}_{\mathcal{C}}$ are calculated according to Lemma \ref{lemma:iteration_proxJ} and Lemma \ref{lemma:iteration_projC}, respectively.

Let $U_{\text{s}}^*$ be the solution of (\ref{eqn:SyncOT_discretized}). The convergence $U_{\text{s}}^{(l)} \rightarrow U_{\text{s}}^*$ is guaranteed \cite{chambolle2011first} with $O(1/N)$ rate of convergence for the primal-dual gap,  provided that $\theta = 1$ and $\tau\sigma \|\mathcal{I}\|^2 < 1$. This has been extended in \cite{banert2023chambolle} to hold when $\theta > \frac{1}{2}$ and $\tau\sigma\|\mathcal{I}\|^2\le \frac{4}{1+2\theta}$. Hereafter, we will fix $\theta = 1$ for the numerical experiments.

\section{Numerical Methods for Synchronized Optimal Transport in Kantorovich Form}\label{section:Numerical_Kantorovich} 

In this section, we present numerical methods for solving SyncOT in Kantorovich form. In this case, the problem is in the form of 
\begin{align}\label{eqn:3terms}
  \min_x f(K(x)) + g(x) + h(x).
\end{align}
Here $K: X\rightarrow Y$ is a continuous linear operator with induced norm 
\[
\|K\| = \max \{ \|Kx\|: \|x\|\le 1 \text{\ for\ } x\in X \}.
\]
$f, g, h$ are proper convex and lower-semicontinuous (l.s.c.). Importantly, we assume that $f, g$ are proximal friendly, and $h$ has a $\beta^{-1}$ Lipschitz continuous gradient. On the other hand, we assume that $h(x)$ is proximally intractable but the gradient is easy to compute or is accessible through auto-differentiation. 

There are several widely used algorithms that are suitable for the problem (\ref{eqn:3terms}). The first algorithm, proposed in \cite{condat2013primal, vu2013splitting}, called Condat-Vu algorithm, generalized the Chambolle-Pock algorithm \cite{chambolle2011first}. Later, a Primal-Dual Fixed-Point algorithm (PDFP) has been proposed in \cite{chen2016primal}, in which two proximal mappings of $g$ are computed in each iteration. PDFP is advantageous of having a wider range of acceptable parameters than Condat-Vu. Most recently, Yan proposed the Yan algorithm in \cite{yan2018new}, which has the same regions of acceptable parameters with PDFP and the same per-iteration complexity as Condat-Vu. 

The above algorithms, Condat-Vu, PDFP, Yan, are all priml-dual based algorithms. Let $(x,\bar{x},s)$ and $(x^+, \bar{x}^+, s^+)$ be the current and the next iterations. The first two steps of Condat-Vu, PDFP and Yan are the same:
\begin{align}
  s^+ &= \text{prox}_{\delta f}(s + \delta K \bar{x}), \label{eqn:3terms_step1}\\
x^+ & = \text{prox}_{\gamma g}(x - \gamma K^* s^+ - \gamma \nabla h(x) ). \label{eqn:3terms_step2}
\end{align}
The key difference lies in the third step of updating $\bar{x}^+$:
\begin{subequations}
  \begin{align}
    \text{Condat-Vu}: \quad &\bar{x}^+ = 2x^+ - x, \label{eqn:3terms_step3_CondatVu} \\
    \text{PDFP}: \quad &\bar{x}^+ = \text{prox}_{\gamma g}(x^+ - \gamma K^* s^+ - \gamma \nabla h(x^+)), \label{eqn:3terms_step3_PDFP} \\
    \text{Yan}: \quad  & \bar{x}^+ = 2x^+ - x + \gamma \nabla h(x) - \gamma \nabla h(x^+). \label{eqn:3terms_step3_Yan}
  \end{align}
\end{subequations}
In summary, (\ref{eqn:3terms_step1}, \ref{eqn:3terms_step2}, \ref{eqn:3terms_step3_CondatVu}) comprise the Condat-Vu algorithm, (\ref{eqn:3terms_step1}, \ref{eqn:3terms_step2}, \ref{eqn:3terms_step3_PDFP}) comprise PDFP, and (\ref{eqn:3terms_step1}, \ref{eqn:3terms_step2}, \ref{eqn:3terms_step3_Yan}) comprise the Yan algorithm. The convergence conditions for the three algorithms are summarized as follows \cite{yan2018new,condat2013primal,vu2013splitting,chen2016primal}:
\begin{subequations}
  \begin{align}
    \text{Condat-Vu}: \quad & \gamma\delta\|KK^*\| + \frac{\gamma}{2\beta} \le 1, \label{eqn:3terms_conv_CondatVu} \\
    \text{PDFP}: \quad & \gamma\delta \|KK^*\|< 1, \ \frac{\gamma}{2\beta} < 1, \label{eqn:3terms_conv_PDFP} \\
    \text{Yan}: \quad  & \gamma\delta \|KK^*\|< 1, \ \frac{\gamma}{2\beta} < 1. \label{eqn:3terms_conv_Yan}
  \end{align}
\end{subequations}
Note that when $h(x)\equiv0$, both Condat-Vu and Yan reduce to the classic Chambolle-Pock \cite{chambolle2011first}.

Among these three algorithms, Yan algorithm has the same per-iteration complexity as the Condat-Vu algorithm (both being faster than PDFP), and shares the same range of acceptable parameters as PDFP (both wider than that of Condat-Vu). Therefore, we will adopt Yan algorithm for our Kantorovich SyncOT problem (\ref{eqn:KantorovichSynOT_02}).

Using the discretization in subsection \ref{subsection:Staggered}, the Kantorovich SyncOT problem (\ref{eqn:KantorovichSynOT_02}) is approximated by solving the following finite-dimensional convex problem over the staggered grid:
\begin{align}\label{eqn:SyncOT_Kantorovich_discretized}
  \min_{U_{\text{s}}\in \mathcal{D}_{\text{s}}} \mathcal{J}(\mathcal{I}(U_{\text{s}})) + \iota_{\mathcal{C}}(U_{\text{s}}) + \mathcal{H}(U_{\text{s}}),
\end{align}
where $\mathcal{J}$ is defined in (\ref{eqn:J_discretized}) in which $J$ is defined in (\ref{eqn:J_definition}) with $A = I$, the set of constraint $\mathcal{C}$ is given in (\ref{eqn:C}), and $\mathcal{H}$ is defined as follows,
\begin{align}\label{eqn: mathcalH}
  \mathcal{H}(U_{\text{s}}) = \mathcal{H}(\rho_{\text{s}}) = \sum_{k=0}^{Q-1} \frac{1}{\Delta t_k} W_2^2\left( \left(\rho_{i,j,k-\frac{1}{2}}\right)_{ij}, \left(\rho_{i,j,k+\frac{1}{2}}\right)_{ij}\right).
\end{align}

Taking
\[
x = U_{\text{s}}\in \mathcal{D}_{\text{s}}, \quad f = \mathcal{J}, \quad K = \mathcal{I}, \quad g = \iota_{\mathcal{C}}, \quad h = \mathcal{H},
\]
Yan algorithm generates a sequence $(U_s^{(l)}, \hat{U}_s^{(l)}, U_c^{(l)} ) \in \mathcal{D}_\text{s} \times \mathcal{D}_\text{s} \times \mathcal{D}_\text{c} $, starting from the given initial $(U_s^{(0)}, \hat{U}_s^{(0)}, U_c^{(0)} ) \in \mathcal{D}_\text{s} \times \mathcal{D}_\text{s} \times \mathcal{D}_\text{c} $ by the following iterations
\begin{align}\label{eqn:PD_Iterations_3terms}
  \begin{cases}
     U_{\text{c}}^{(l+1)} = \text{prox}_{\tau\mathcal{J}^*} (U_{\text{c}}^{(l)} + \tau \mathcal{I}(\bar{U}_{\text{s}}^{(l)}) ), \\
     U_{\text{s}}^{(l+1)} = \text{proj}_{\mathcal{C}} \left(U_{\text{s}}^{(l)} - \sigma \mathcal{I}^*(U_{\text{c}}^{(l+1)}) - \sigma\nabla\mathcal{H}(U_{\text{s}}^{(l)}) \right), \\
     \bar{U}_{\text{s}}^{(l+1)} = 2U_{\text{s}}^{(l+1)} - U_{\text{s}}^{(l)} + \sigma\nabla\mathcal{H}(U_{\text{s}}^{(l)}) - \sigma\nabla\mathcal{H}(U_{\text{s}}^{(l)}).
  \end{cases}
\end{align}
Similar to the discussion in Section \ref{subsection:Primal_dual}, $\text{prox}_{\tau\mathcal{J}^*}$ is calculated by Moreau identity (\ref{eqn:Moreau}). Furthermore, $\text{prox}_{\tau\mathcal{J}}$ and $\text{proj}_{\mathcal{C}}$ are calculated according to Lemma \ref{lemma:iteration_proxJ} and Lemma \ref{lemma:iteration_projC}, respectively. Importantly, $\nabla \mathcal{H}(U_{\text{s}
}^{(l)})$ is computed via auto-differentiation.

\begin{remark}

By the convergence theory of Yan algorithm \cite{{yan2018new}}, $\mathcal{H}$ in (\ref{eqn: mathcalH}) is required to be a convex function with $\beta^{-1}$ Lipschitz gradient. The convexity of $\mathcal{H}(\cdot)$ is evident. Indeed, since
\begin{align*}
  &\mathcal{H}(\lambda_1\rho^{(1)}_{\text{s}}+ \lambda_2\rho^{(2)}_{\text{s}}) \\
  =& \sum_{k=0}^{Q-1} \frac{1}{\Delta t_k} W_2^2\left( \lambda_1 \left(\rho^{(1)}_{i,j,k-\frac{1}{2}}\right)_{ij} + \lambda_2 \left(\rho^{(2)}_{i,j,k-\frac{1}{2}}\right)_{ij}, \lambda_1\left(\rho^{(1)}_{i,j,k+\frac{1}{2}}\right)_{ij} + \lambda_2\left(\rho^{(2)}_{i,j,k+\frac{1}{2}}\right)_{ij}\right),
\end{align*}
for $\lambda_1, \lambda \in [0,1]$ and $\lambda_1+\lambda_2 = 1$, we only need to verify that $W_2^2(a,b)$ is jointly convex. To this end, consider $W_2^2(a^{(1)},b^{(1)})$ and $W_2^2(a^{(2)},b^{(2)})$, and assume that $P^{(i)}$ are optimal couplings between $a^{(i)}$ and $b^{(i)}$ for $i=1,2$. Then 
\begin{align*}
  & \ W_2^2(\lambda_1 a^{(1)}+ \lambda_2 a^{(2)}, \lambda_1 b^{(1)}+ \lambda_2 b^{(2)} ) \\
  \le \ & \left\langle \lambda_1 P^{(1)} + \lambda_2 P^{(2)}, C \right\rangle \\
  = \ & \lambda_1 \langle  P^{(1)}, C \rangle + \lambda_2 \langle  P^{(2)}, C \rangle \\
  = \ &\lambda_1 W_2^2(a^{(1)},b^{(1)}) + \lambda_2 W_2^2(a^{(2)},b^{(2)}).
\end{align*}
Therefore it follows the joint convexity of $W_2^2(\cdot,\cdot)$, and therefore the convexity of $\mathcal{H}(\cdot)$.

The gradient of $\mathcal{H}(\cdot)$ is due to the fact that:
\begin{align}
  \nabla W_2^2(a,b) = 
  \begin{bmatrix*}[r]
    f^* \\
    g^* \\
\end{bmatrix*},
\end{align}
in which $f^*$ and $g^*$ are the optimal Kantorovich potentials chosen so that their coordinates sum to 0. Then it follows that
\begin{align}\label{eqn:gradient_H}
  \nabla \mathcal{H}(\rho_{\text{s}}) = \left( f^*_{i,j,k-\frac{1}{2}} \right)_{(i,j,k)\in\mathcal{S}_{\text{s}}^t}.
\end{align}
Here for $k = 1, \cdots, Q-1$, 
\begin{align}
  f^*_{i,j,k-\frac{1}{2}} = \bar{f}^*_{i,j,k-\frac{1}{2}} + \hat{f}^*_{i,j,k-\frac{1}{2}},
\end{align}
where
\begin{align*}
  \left( \bar{f}^*_{i,j,k-\frac{1}{2}} \right)_{ij} & = \nabla_{a} W_2^2\left( \left(\rho_{i,j,k-\frac{3}{2}}\right)_{ij}, \left(\rho_{i,j,k-\frac{1}{2}}\right)_{ij}\right), \quad \text{with\ } \sum_{i,j} \bar{f}^*_{i,j,k-\frac{1}{2}} = 0, \\
  \left( \hat{f}^*_{i,j,k-\frac{1}{2}} \right)_{ij} & = \nabla_{a} W_2^2\left( \left(\rho_{i,j,k-\frac{1}{2}}\right)_{ij}, \left(\rho_{i,j,k+\frac{1}{2}}\right)_{ij}\right), \quad \text{with\ } \sum_{i,j} \hat{f}^*_{i,j,k-\frac{1}{2}} = 0,
\end{align*}
and for $k=0$ and $k=Q$, 
\begin{align*}
  \left( f^*_{i,j,-\frac{1}{2}} \right)_{ij} & = \nabla_{a} W_2^2\left( \left(\rho_{i,j,-\frac{1}{2}}\right)_{ij}, \left(\rho_{i,j,\frac{1}{2}}\right)_{ij}\right), \quad \text{with\ } \sum_{i,j} f^*_{i,j,-\frac{1}{2}} = 0,\\
  \left( f^*_{i,j,Q-\frac{1}{2}} \right)_{ij} & = \nabla_{a} W_2^2\left( \left(\rho_{i,j,Q-\frac{3}{2}}\right)_{ij}, \left(\rho_{i,j,Q-\frac{1}{2}}\right)_{ij}\right), \quad \text{with\ } \sum_{i,j} f^*_{i,j,Q-\frac{1}{2}} = 0.
\end{align*}
Here $\nabla_a W_2^2(\cdot,\cdot)$ represents the gradient with respect to the first variable.

On the other hand, since the gradient of $\mathcal{H}(\rho_{\text{s}})$ in (\ref{eqn:gradient_H}) depends on $\rho_{\text{s}}$ in an implicit way, it remains unclear whether $\nabla \mathcal{H}(\rho_{\text{s}})$ is $\beta^{-1}$ Lipschitz continuous. Therefore, the convergence of Yan algorithm (\ref{eqn:PD_Iterations_3terms}) when applied to the Kantorovich SyncOT problem (\ref{eqn:KantorovichSynOT_02}) is currently only demonstrated by empirical and numerical evidence.

\end{remark}

\section{Numerical Experiments}\label{section:Numerical}

In this section, we present several numerical experiments to demonstrate the SyncOT model framework in both Monge and Kantorovich forms.

\subsection{Spaces connected through Monge map}

\begin{figure}[p]
  \centering
  \includegraphics[width=0.75\textwidth]{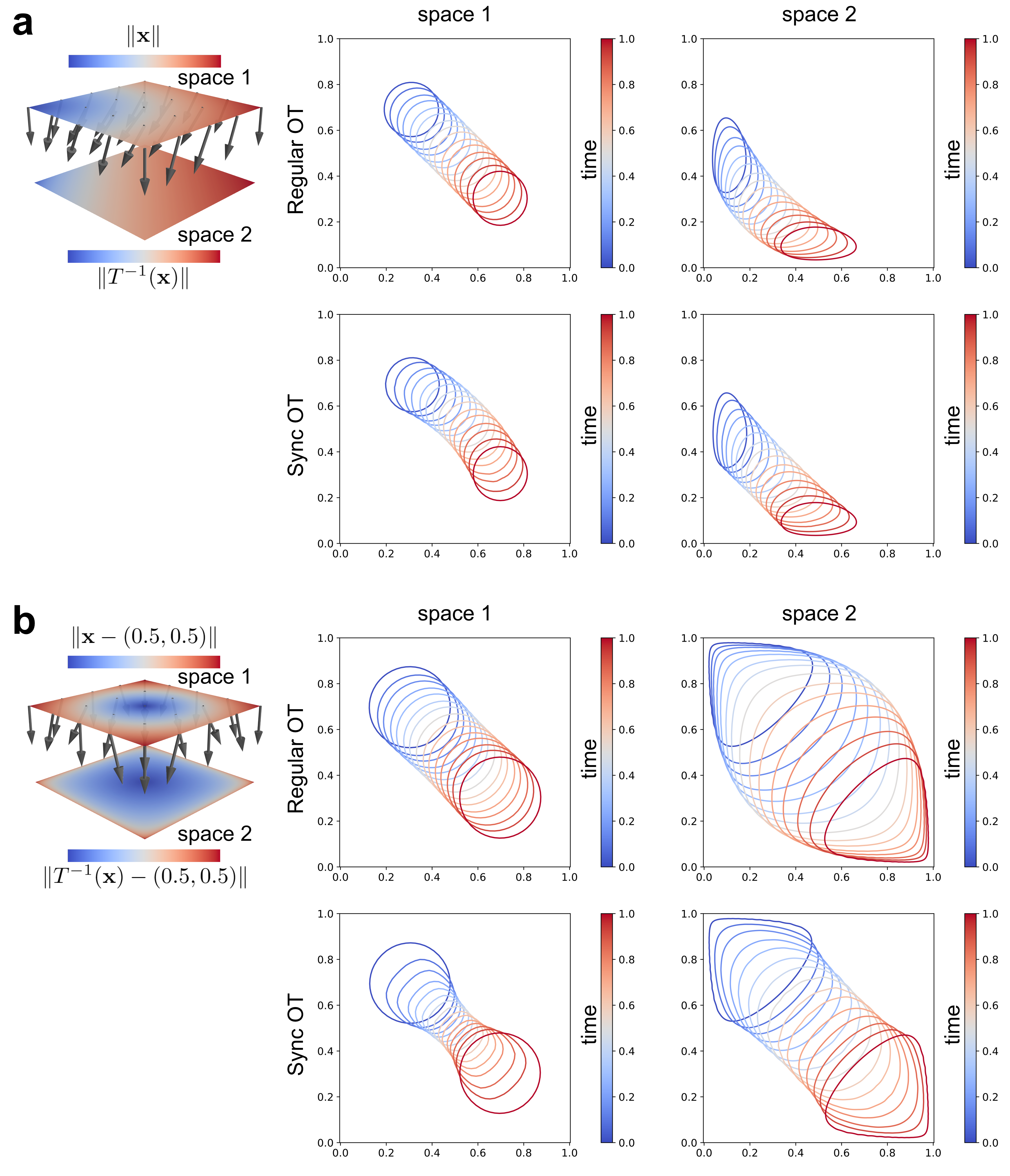} 
  \caption{SyncOT in Monge form with $X$ and $Y$ in $\mathbb{R}^2$. The spaces connected by $\mathbf{T}(x,y)=(x^2,y^2)$ in (a) and $\mathbf{T}(x,y)=\left(\frac{1}{1+e^{-10(x-0.5)}},\frac{1}{1+e^{-10(y-0.5)}}\right)$ in (b). The parameters $\alpha_1=0.05$ and $\alpha_2=0.95$ were used for SyncOT in both examples. The evolving densities $\rho_t$ are visualized as levelsets of $0.5(\mathrm{min}\rho_t+\mathrm{max}\rho_t)$. The maps are visualized with matching colors in spaces 1 and 2.}
  \label{fig:numerical_02}
\end{figure}

In this section, we illustrate SyncOT with examples where spaces are connected by Monge maps. In the numerical examples, we consider the unit square as the computation domain and the marginal distributions are generated with normalized Gaussian distribution truncated to the unit square,
\[
  N_{[0,1]^2}(x,y;x_0,y_0,\sigma) = \frac{p(x,y;x_0,y_0,\sigma)}{\int_{[0,1]^2}p(x,y;x_0,y_0,\sigma)\mathrm{d}x\mathrm{d}y},
\]
where $p(x,y;x_0,y_0,\sigma)=e^{-\frac{(x-x_0)^2+(y-y_0)^2}{2\sigma^2}}$.

First, we consider maps $\mathbf{T}:\mathbb{R}^{2}\rightarrow\mathbb{R}^2$ in Figure \ref{fig:numerical_02}. In Figure \ref{fig:numerical_02}a, the problem is set up with 
\begin{equation}
  \rho_0 = N_{[0,1]^2}(\cdot,\cdot;0.3,0.7,0.1), \quad\rho_1 = N_{[0,1]^2}(\cdot,\cdot;0.7,0.3,0.1),\quad \mathbf{T}(x,y)=(x^2,y^2),
\end{equation}
where the Monge map is induced by $\mathbf{T}$.
Solving regular OT on the first space $X$ resulted in an optimal straight path in $X$, which corresponds to a nonoptimal curved path in the second space $Y$. To illustrate the idea of SyncOT, we applied a large weight to the second space by setting $\alpha_1=0.05, \alpha_2=0.95$ in (\ref{eqn:SynOT_2Spaces}). Shown in Figure \ref{fig:numerical_02}a, this led to a suboptimal curved path in $X$ but a near-optimal straight path in $Y$. Similarly, in Figure \ref{fig:numerical_02}b, we took the same marginal distributions as in Figure \ref{fig:numerical_02}a, with the Monge map induced by
\begin{equation}
  \mathbf{T}(x,y)=\left(\frac{1}{1+e^{-10(x-0.5)}},\frac{1}{1+e^{-10(y-0.5)}}\right),
\end{equation}
which pushes the mass on the diagonal towards the corners $(0,0)$ and $(1,1)$. While regular OT and SyncOT both led to straight paths of mass center in $X$ and $Y$, regular OT preserves the shape in $X$, but the corresponding path in $Y$ undergoes expansion and contraction. In SyncOT, the size of the evolving density stays much more consistent in $Y$ compared to that in $X$, as shown in Figure \ref{fig:numerical_02}b.

\begin{figure}[p]
  \centering
  \includegraphics[width=0.8\textwidth]{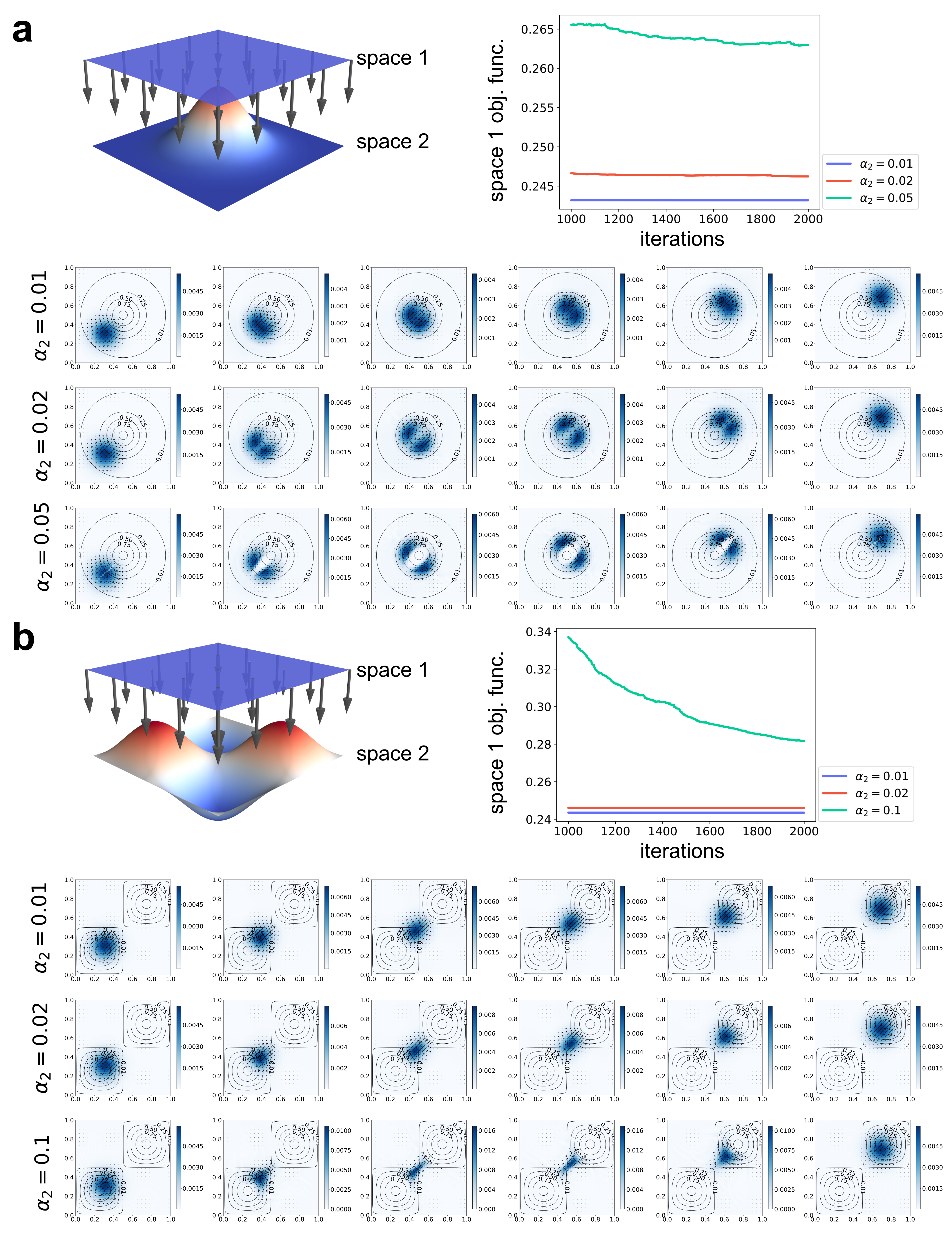} 
  \caption{SyncOT in Monge form with $X\subset\mathbb{R}^2$ and $Y\subset\mathbb{R}^3$. In both cases, $\rho_0=N_{[0,1]^2}(\cdot,\cdot;0.3,0.3,0.1)$ and $\rho_1=N_{[0,1]^2}(\cdot,\cdot;0.7,0.7,0.1)$. The Monge maps are $\mathbf{T}(x,y,z)=(x,y,e^{-\frac{(x-0.5)^2+(y-0.5)^2}{2\sigma^2}}), \sigma=0.15$ and $\mathbf{T}(x,y)=(x,y,\sin(2\pi x)\sin(2\pi y))$ for (a) and (b), respectively. The evolving density is visualized as heatmaps in $X$ and space $Y$ is visualized as contour lines. The transport costs in space $X$ over the optimization iterations are shown for different values of the parameter $\alpha_2$.}
  \label{fig:numerical_03}
\end{figure}

We further illustrate SyncOT with maps $\mathbf{T}: \mathbb{R}^2\rightarrow\mathbb{R}^3$, where we consider $X=[0,1]^2$ and $Y=\{(x,y,z): z=z(x,y),\,\mathrm{for} (x,y)\in X\}$. In both examples in Figure \ref{fig:numerical_03}, the marginal distributions were taken to be $\rho_0=N_{[0,1]^2}(\cdot,\cdot;0.3,0.3,0.1)$ and $\rho_1=N_{[0,1]^2}(\cdot,\cdot;0.7,0.7,0.1)$. In Figure \ref{fig:numerical_03}a, the map $z(x,y)=e^{-\frac{(x-0.5)^2+(y-0.5)^2}{2\sigma^2}}$ with $\sigma=0.15$. With an increasing weight for the objective function in $Y$, from $\alpha_2=0.01, 0.02$ to $0.05$, the pattern becomes clearer where the path circumvents the central bump. As illustrated by the objective function in $X$ that larger value of $\alpha_2$ led to higher cost in $X$, the lower cost path in $Y$ was achieved by trading off the path optimility in $X$. In another example shown in Figure \ref{fig:numerical_03}b, a similar property was observed as expected. Here, $Y$ is induced by $z(x,y)=\sin(2\pi x)\sin(2\pi y)$. SyncOT led to a distribution with shrinked size at the midpoint of the transport, so that the distribution stayed on the ridge between the two bumps to minimize the transport cost in $Y$. This behavior became increasingly evident and the cost in $X$ increased as the weight $\alpha_2$ increased.

\begin{figure}[p]
  \centering
  \includegraphics[width=0.8\textwidth]{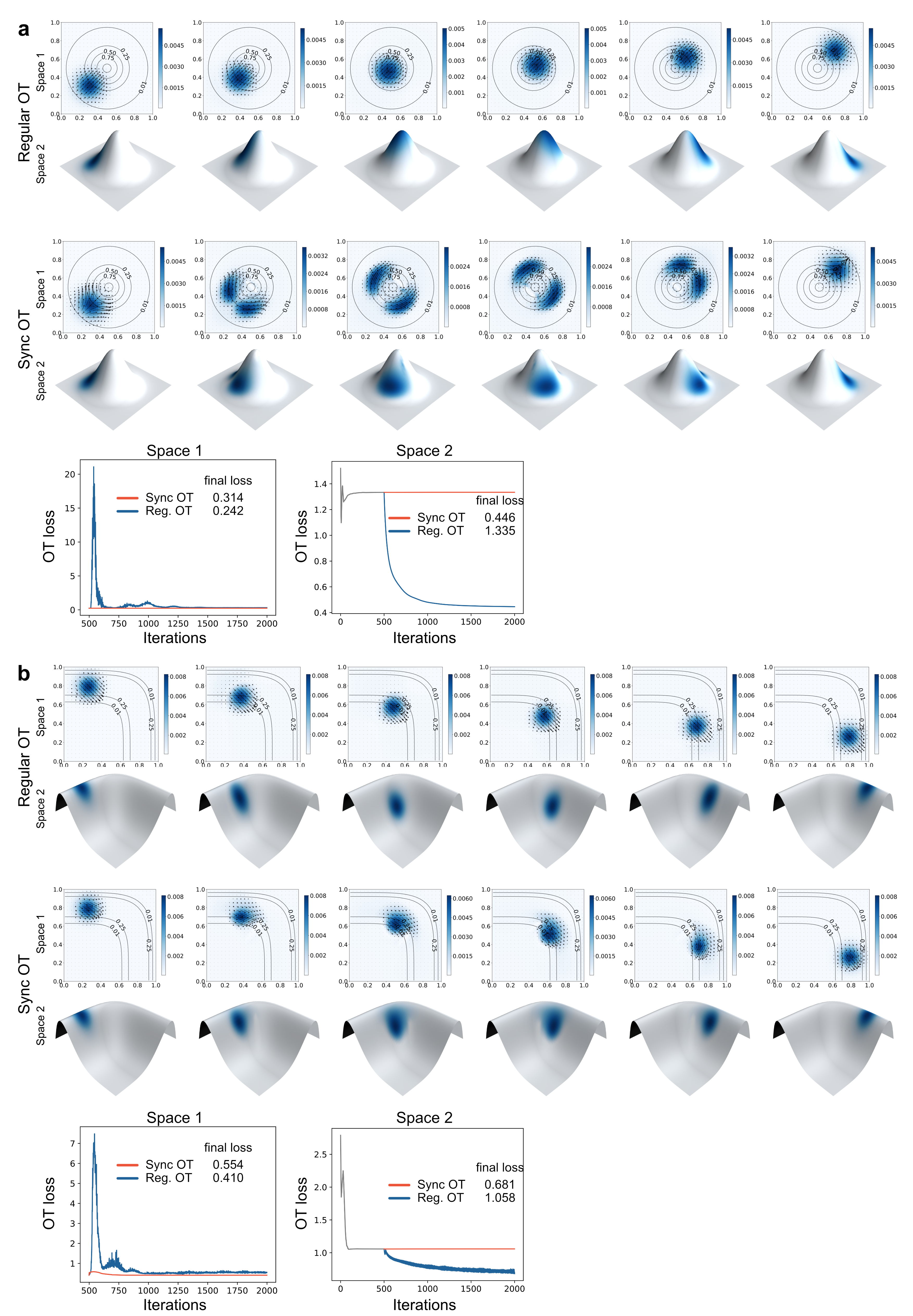} 
  \caption{SyncOT solved in Kantorovich form for $X\subset\mathbb{R}^2$ and $Y\subset\mathbb{R}^3$ (a) $\rho_0=N_{[0,1]^2}(\cdot,\cdot;0.3,0.3,0.1)$, $\rho_1=N_{[0,1]^2}(\cdot,\cdot;0.7,0.7,0.1)$, $\mathbf{T}(x,y,z)=(x,y,e^{-\frac{(x-0.5)^2+(y-0.5)^2}{2\sigma^2}}), \sigma=0.15$. (b) $\rho_0=N_{[0,1]^2}(\cdot,\cdot;0.25,0.8,0.08)$, $\rho_1=N_{[0,1]^2}(\cdot,\cdot;0.8,0.25,0.08)$, $\mathbf{T}(x,y) = (x,y,-0.5\cos(5\sqrt{x^5+y^5}))$. The evolving densities are shown in both spaces where space $Y$ is visualized as contours, and space $X$ as heatmaps. The transport cost in both spaces over iterations are visualized.}
  \label{fig:numerical_04}
\end{figure}

\subsection{Spaces connected through Kantorovich coupling}
In this section, we demonstrate the numerical solvers designed for problem (\ref{eqn:SyncOT_Kantorovich_discretized}) when the correspondence between two spaces is described in the Kantorovich form. Figure \ref{fig:numerical_04}a revisits the same problem initially presented in Figure \ref{fig:numerical_03}a, showing that identical results can be achieved using the Kantorovich form instead of the Monge form in numerical implementation. Specifically, the joint distribution $\pi\in\mathcal{M}_+ (X\times Y)$ induced by $\mathbf{T}:X\rightarrow Y$ is defined as 
\begin{equation*}
  \pi(A) = \mu(\{\mathbf{x}\in X: (\mathbf{x}, \mathbf{T}(\mathbf{x}))\in A\}).
\end{equation*}
In Figure \ref{fig:numerical_04}b, we consider $\rho_0=N_{[0,1]^2}(\cdot,\cdot;0.25,0.8,0.08)$ and $\rho_1=N_{[0,1]^2}(\cdot,\cdot;0.8,0.25,0.08)$. The second space $Y$ is induced by $\mathbf{T}: (x,y)\mapsto (x,y,-0.5\cos(5\sqrt{x^5+y^5}))$. While this problem could be solved using the Monge form, we implemented it in Kantorovich form to illustrate the efficiency of the numerical algorithm. Specifically, space $X$ is discretized into regular grids and $Y$ is discretized by applying the map $\mathbf{T}$ to the grid points in $X$. The Kantorovich map was then represented by an identity matrix in this case. In this example, the optimal path in $X$ corresponds to a path in $Y$ that descents into the valley and ascent back to the ridge. SyncOT resulted in a shorter path in $Y$ that stays near the ridge region during transport, but a longer and slightly curved path in $X$.

\begin{figure}[p]
  \centering
  \includegraphics[width=0.8\textwidth]{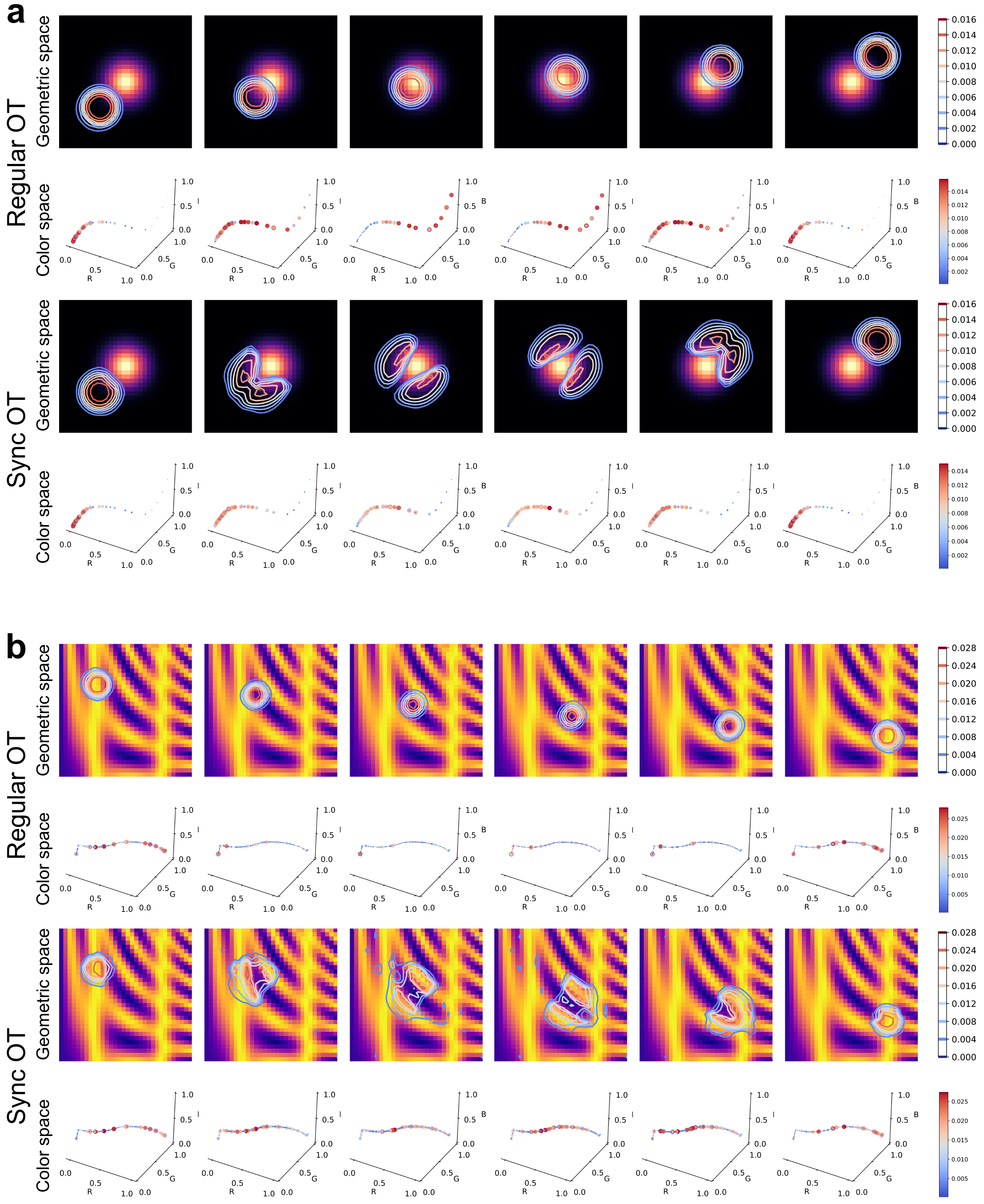} 
  \caption{SyncOT addressing consistency in background environment. In both experiments, space $X$ is the spatial space which is a unit square, and space $Y$ is a one-dimensional curve in the three-dimensional RGB space describing the background colors of the spatial space $X$. The evolving densities are visualized as contours in space $X$ and as weighted point clouds in space $Y$, where the weight is described by point color and size. The synthetic examples are generated by applying a scalar colormap from the package Matplotlib to a predefined function. The color maps {\it magma} and {\it cividis} and the functions $f(x,y)=e^{-\frac{(x-0.5)^2+(y-0.5)^2}{0.15^2}}$ and $f(x,y)=\sin^{10}(6x)+\cos(10+36xy)\cos(6x)$ were used for (a) and (b), respectively.  }
  \label{fig:numerical_05}
\end{figure}

Finally, we illustrate the SyncOT solver in a scenario commonly encountered in biological system modeling, where the transported mass tends to stay within a specific environment, such as migrating cells that prefer a certain microenvironment. To effectively demonstrate the numerical algorithm, we abstract this problem into a conceptual example where space $X$ is equipped with a certain background color, and the color resembles $Y$ as a subspace within the three-dimensional RGB color space. In these examples, SyncOT resulted in paths that stayed in a consistent environment, as depicted by the background colors during transportation; specifically, the darker and brighter regions in Figure \ref{fig:numerical_05}a and \ref{fig:numerical_05}b, respectively. In both examples, SyncOT achieved lower transport costs in the color space by taking paths with longer spatial distances.

\section{Conclusion and Discussion}

In this work, we introduce a framework to model the synchronized optimal transport cross multiple spaces. Two types of correspondence between the spaces, the Monge form and the Kantorovich form, are considered. The SyncOT aims to minimize the total cost over all spaces. For the Monge SyncOT, we show that it can be recast into a single-space OT problem with a generalized kinetic energy in terms of the Monge map. The Kantorovich SyncOT problem is reformulated using the Riemann sum approximation for practical implementation. We have developed primal-dual based numerical algorithms to solve both the Monge and Kantorovich SyncOT problems. The properties of SyncOT and the effects of various parameter values are illustrated through numerical examples. 

This work can be further extended in several directions. Firstly, efficient numerical algorithms for high dimensions are needed to apply SyncOT to high-dimensional data such as single-cell multi-omics data. Recently, deep learning based approaches have been developed to solve high-dimensional dynamical OT problem \cite{neklyudov2023computational}, which can potentially be extended to solve high-dimensional SyncOT problems. Secondly, in practice, static optimal transport or neural network could be used to find and represent the correspondence between the spaces. In the cases when the correpondence between spaces can not be determined in advance, it would be useful to develop a more flexible framework of SyncOT to optimize for the maps between the spaces as well. Lastly, this work currently considers balanced optimal transport. An important future extension is to develop unbalanced SyncOT for applications where it is needed to allow the mass to grow or decay during transportation.

\section*{Acknowledgements}

Z. Cang's work is supported by a startup grant from North Carolina State University and NSF grant DMS2151934. Y. Zhao's work is supported by a grant from the Simons Foundation through Grant No. 357963 and NSF grant DMS2142500.

\end{document}